\documentclass[11pt,a4paper,american,reqno]{amsart}
\usepackage{amssymb}
\usepackage{amsmath}
\usepackage{graphicx}
\usepackage{enumerate}
\usepackage{physics}
\usepackage{amssymb}
\usepackage{amsthm}
\usepackage{epsf}
\usepackage{amsbsy,amsmath}
\usepackage{mathtools}
\usepackage{mathrsfs}
\usepackage{amsfonts}
\usepackage{eucal}
\usepackage{graphics,mathrsfs}
\usepackage{amsthm}
\usepackage{secdot}
\usepackage{esint}
\usepackage{bbm}
\usepackage{varwidth}
\usepackage{tasks}
\usepackage{cite}
\usepackage[colorlinks,pdfpagelabels,pdfstartview = FitH,bookmarksopen = true,bookmarksnumbered = true,linkcolor = red,plainpages = false,hypertexnames = false,citecolor = red,pagebackref=false]{hyperref}

\theoremstyle{plain}
\newtheorem{theorem}{Theorem}[section]

\newtheorem{lemma}[theorem]{Lemma}

\theoremstyle{definition}
\newtheorem{definition}[theorem]{Definition}

\newtheorem{remark}[theorem]{Remark}
\newtheorem{notation}[theorem]{Notation}

\allowdisplaybreaks

\usepackage{xcolor}

\long\def\symbolfootnote[#1]#2{\begingroup
\def\thefootnote{\fnsymbol{footnote}}\footnote[#1]{#2}\endgroup}

\numberwithin{equation}{section}
\begin{document}

\title[Optimization problem for local-nonlocal operators with potential]{Optimization for the first weighted eigenvalue of local-nonlocal operators with a potential}
\author{R. Lakshmi, Sekhar Ghosh and Ratan Kr. Giri}
\address[R. Lakshmi]{Department of Mathematics, National Institute of Technology Calicut, Kozhikode, Kerala, India - 673601}
\email{lakshmir1248@gmail.com / lakshmi\_p220223ma@nitc.ac.in}
\address[Sekhar Ghosh]{Department of Mathematics, National Institute of Technology Calicut, Kozhikode, Kerala, India - 673601}
\email{sekharghosh1234@gmail.com / sekharghosh@nitc.ac.in}
\address[Ratan Kr. Giri]{Department of Mathematics, The LNM Institute of Information Technology, Jaipur, India - 302031
}
\email{giri90ratan@gmail.com / ratan.giri@lnmiit.ac.in }


\begin{abstract} 
 In this work, we study the optimization problem for the first eigenvalue of mixed local nonlocal operators plus a potential $V$. We begin by investigating the existence and fundamental properties of the eigenvalues, with special emphasis on the first eigenvalue. Finally, we discuss the dependence of the first eigenvalue on the potential function and establish the existence of optimal potentials within certain admissible classes. In particular, we show the existence of a unique maximizer and a minimizer of the first eigenvalue on any bounded, closed, and convex subset of $L^q(\Omega)$. Moreover, these results enable us to characterize the maximizers and minimizers in the closed unit ball of $L^q(\Omega)$, as well as in the class of rearrangements of any $V\in L^q(\Omega)$. 
\end{abstract}

\maketitle
\section{Introduction}
\noindent  The analysis of eigenvalue problems for elliptic operators plays a fundamental role in the study of partial differential equations, owing to their strong influence on the existence and qualitative properties of solutions to elliptic PDEs, as well as their wide applications in mathematical physics \cite{BH17, G18, BS2022}. Over the years, the existence and properties of eigenvalues for various local, nonlocal, and mixed operators have been extensively investigated; see, for example, \cite{L1990, AP1987, LL2014, FP2014, DFR2019, CQ2009} and the references therein.
\par The optimization theory of eigenvalues of elliptic operators is a significant problem in the theory of eigenvalues due to their importance in fields such as scattering theory, quantum mechanics, etc., (see \cite{BS2007, BS2022}). One of the earlier works in the optimization of eigenvalues to elliptic operators is by Harrell \cite{H1984}, who studied the existence and properties of a potential $V$ which maximizes the first eigenvalue for the operator $-\Delta+V$ on a set of potentials with a restriction on the $L^p$ norm of $V$. Later, for the same operator $-\Delta+V$ with some restrictions on $V$, the existence and uniqueness of both maximizers and minimizers to the first eigenvalue were studied by Ashbaugh and Harrell \cite{AH1987}, and the optimizers were characterized.  Subsequently, Bonder and Pezzo \cite{FD2006} extended the results in \cite{AH1987} to the first eigenvalue for the operator $-\Delta_p+V$, where $\Delta_p$ is the $p$-Laplacian defined by
    $$\Delta_p u=\operatorname{div}(|\nabla u|^{p-2}\nabla u).$$
    In the fractional setup, the study on the optimization of the first eigenvalue to the operator $(-\Delta_p)^s+V$, where $(-\Delta_p)^s$ is the fractional $p$-Laplacian given by
    $$(-\Delta_p)^s u(x)=\int_{\mathbb{R}^N}\frac{|u(x)-u(y)|^{p-2}(u(x)-u(y))}{|x-y|^{N+sp}}dy,$$
    has been explored by Pezzo \textit{et al.} \cite{PBR2018}.
\par The optimization theory of weighted eigenvalues to elliptic operators has also been widely investigated in a set of rearrangements of a potential, weight, or both, where the set of rearrangements of any function $V$, is defined via equimeasure as follows.
\begin{equation}\label{F-RV}
    \mathcal{R}_V:=\left\{W\in L^q(\Omega): |\{W\geq t\}|=|\{V\geq t\}|, \text{ for all } t\in \mathbb{R}\right\}.
\end{equation}
The optimization of the first weighted eigenvalue for the Laplacian with certain restrictions on the weight has been explored in \cite{CM1990, K1955}. Cuccu \textit{et al.} \cite{CEP2009} have investigated the optimization problem of the first weighted $p$-Laplacian with weight in a set of rearrangements of a fixed function $g_0$. For the $p$-Laplacian with potential, the minimization of the first weighted eigenvalue over $\mathcal{R}_{V_0} \times \mathcal{R}_{g_0}$, for a fixed potential $V_0$ and weight $g_0$ was initially studied by Pezzo and Bonder \cite{PB2010} with certain assumptions on $V_0$ and $g_0$, and the study was later extended to more general assumptions on $V_0$ and $g_0$ by Biswas \textit{et al.} \cite{BDG2023}. The optimization of the principal eigenvalue of the weighted $p$-Laplacian with the weights taken from the characteristic functions of subsets of the domain with a fixed measure $\alpha$ has been explored by Marras \textit{et al.} \cite{MPV2013}. The authors in \cite{MPV2013} also studied several properties of the optimizing eigenvalue depending on $\alpha$. The minimization problem for the first weighted eigenvalue of the fractional Laplacian with potential, on a set of rearrangements of the potential and weight, has also been explored by Ghosh \cite{G2025}.
\par The eigenvalue problem involving the mixed local and nonlocal operator plus a potential, $-\Delta_p+(-\Delta_p)^s+V$ in a bounded domain with non-negative potential $V$ has been recently explored in Lakshmi, Giri and Ghosh \cite{LGG2026} for the case $sp<N$. In this work, we extend the results in \cite{LGG2026, PBR2018} to the eigenvalue problem of the mixed local and non-local operator with sign-changing potential, and for all $0<s<1<p<\infty$. In particular, we study the optimization of the first eigenvalue $\lambda_1(V)$ to the problem
\begin{align}\label{F}
    -\Delta_p u+(-\Delta_p)^s u+V|u|^{p-2}u&=\lambda |u|^{p-2}u \hspace{.2cm}\text{ in } \Omega,\nonumber\\
    u&=0 \hspace{1.5cm}\text{ in } \mathbb{R}^N \setminus \Omega,
\end{align}
where $\Omega$ is a bounded domain, $0<s<1<p<\infty$ and the potential $V \in L^q(\Omega)$ for some $q \in \left(\max\{\frac{N}{p},1\},\infty\right)$. Here $\Delta_p$ is the $p$-Laplacian and $(-\Delta_p)^s$ denotes the operator
$$(-\Delta_p)^s u (x)=\int_{\mathbb{R}^N}|u(x)-u(y)|^{p-2}(u(x)-u(y))K_{s,p}(x,y)dy,$$
where the kernel $K_{s,p}$ is a function that is comparable to the map $(x,y)\mapsto \frac{1}{|x-y|^{N+sp}}$. The reader may refer to Section \ref{F-s1} for the properties of the kernel $K_{s,p}$ in detail. We first discuss the existence and properties of the eigenvalues and eigenfunctions, which are summarised in the following theorem.
\begin{theorem}
    For the problem \eqref{F} with potential $V\in L^q(\Omega)$, the following statements hold.
    \begin{enumerate}[(i)]
        \item There exists a smallest eigenvalue $\lambda_1(V)\in \mathbb{R}$ for \eqref{F}.
        \item For any eigenvalue $\lambda$ of \eqref{F}, the associated eigenfunctions are bounded in $\mathbb{R}^N$.
        \item For any eigenfunction $u$ associated with $\lambda_1(V)$, either $u>0$ or $u<0$ in $\Omega$. Moreover, eigenfunctions associated with eigenvalues higher than $\lambda_1(V)$ change sign in $\Omega$.
        \item The first eigenvalue $\lambda_1(V)$ is simple and isolated.
    \end{enumerate}
\end{theorem}
Now, we state our first main result obtained in the paper.
\begin{theorem}\label{F-T1}
    Let $\mathcal{A}\subset L^q(\Omega)$ be bounded, closed, and convex. Then, there exists a minimizer $\underline{V}$ of the functional $\lambda_1(\cdot)$ in $\mathcal{A}$. Also, if $\underline{u}$ is the eigenfunction corresponding to $\lambda_1(\underline{V})$ of the problem \eqref{F} with potential $\underline{V}$, satisfying $\|\underline{u}\|_{L^p(\Omega)}=1$, then $\underline{V}$ is the unique minimizer of the functional
    $$V \mapsto \int_{\Omega}V|\underline{u}|dx, \ V\in \mathcal{A}.$$
\end{theorem}
We also discuss the following result on the maximizers of the first eigenvalue to \eqref{F}.
\begin{theorem}\label{F-T2}
    Let $\mathcal{A}\subset L^q(\Omega)$ be bounded, closed, and convex. Then, the maximum of $\lambda_1(\cdot)$ in $\mathcal{A}$ is achieved by a unique maximizer $\overline{V}\in \mathcal{A}$.
\end{theorem}
We have also investigated the properties of eigenvalues and eigenfunctions to \eqref{F} and characterized the optimizers on certain subsets of $L^q(\Omega)$. 

The paper is structured as follows. In Section \ref{F-s1}, we introduce some of the concepts and definitions required to study the eigenvalues to \eqref{F}. In Section \ref{F-s2}, we prove the existence and some properties of the eigenvalues and associated eigenfunctions to \eqref{F} for any fixed potential $V\in L^q(\Omega)$. In Section \ref{F-s3}, we prove the main results Theorem \ref{F-T1} and Theorem \ref{F-T2} along with the study of certain characteristics of $\lambda_1(\cdot)$. Then, in Section \ref{F-s4}, we analyze the properties of optimizers in a set of rearrangements of a function $V\in L^q(\Omega)$ by equimeasure, and its weak closure in $L^q(\Omega)$. Finally, we conclude the paper with a study of the optimization problem of $\lambda_1(\cdot)$ in a closed unit ball in $L^q(\Omega)$.

\section{Preliminaries and notations}\label{F-s1} 
\noindent In this section, we discuss a few definitions, notations, and properties that are crucial for studying the problem \eqref{F}. We begin by presenting the definition of the spaces $W_0^{1,p}(\Omega)$ and $W_0^{1,p}(\Omega)$ (see \cite{DD2012}). For an open set $\Omega\subset \mathbb{R}^N$ with boundary $\partial \Omega$ and $0<s<1\leq p<\infty$, we have
\begin{align*}
    W_0^{1,p}(\Omega)&=\{u\in W^{1,p}(\Omega): u=0 \text{ on }\partial\Omega\},\\
    \text{ and }W_0^{s,p}(\Omega)&=\{u\in W^{s,p}(\Omega): u=0 \text{ on }\mathbb{R}^N \setminus \Omega\}.
\end{align*}
Here, the spaces $W^{1,p}(\Omega)$ and $W^{s,p}(\Omega)$ denote the Sobolev and fractional Sobolev space respectively. The reader may refer to \cite{DD2012, NPV2012, B2011, L2023} for the definition and properties of $W^{1,p}(\Omega)$ and $W^{s,p}(\Omega)$. We discuss a theorem that plays an important role later in the characterisation of our solution space.
\begin{theorem}\label{F-SFE}(See \cite[Lemma 2.1]{BDD2022})
Let $\Omega\subset \mathbb{R}^N$ be a bounded, Lipschitz domain and $0< s<1<p<\infty$. Then, for any $u\in W_0^{1,p}(\Omega)$, the following inequality holds.
\begin{equation*}
\int_{\mathbb{R}^N}\int_{\mathbb{R}^N}\frac{| \tilde{u}(x)-\tilde{u}(y)|^p}{| x-y |^{N+ps}} dxdy \leq C \int_\Omega |\nabla u|^p dx.
\end{equation*}
Here, $C= C(N,s,p,\Omega)>0$ is a constant and $\tilde{u}$ is the function obtained by extending $u$ by zero to $\mathbb{R}^N \setminus \Omega$.
\end{theorem}
\par Now, we state the assumptions on the Kernel $K_{s,p}$. The function $K_{s,p}:\mathbb{R}^N \times \mathbb{R}^N \rightarrow \mathbb{R}$ is taken to be a measurable function satisfying the following.
\begin{enumerate}
    \item $K_{s,p}$ is translation invariant. i.e., $K_{s,p}(x+z,y+z)=K_{s,p}(x,y)$ for all $x,y,z \in \mathbb{R}^N$.
    \item $K_{s,p}$ is symmetric. i.e., $K_{s,p}(x,y)=K_{s,p}(y,x)$ for all $x,y \in \mathbb{R}^N$.
    \item There exists a constant $\Lambda>0$ such that
    \begin{equation}\label{F-Kernel}
        \frac{1}{\Lambda|x-y|^{N+sp}}\leq K_{s,p}(x,y)\leq \frac{\Lambda}{|x-y|^{N+sp}}, \text{ for all } x,y \in \mathbb{R}^N, \ x \neq y.
    \end{equation}
    \item $K_{s,p}(x,\cdot)$ is continuous in $\mathbb{R}^N \setminus \{x\}$.
\end{enumerate}
\par Finally, we move to the definition of the solution space for the problem \eqref{F}. From now on, we assume $\Omega \subset \mathbb{R}^N$ is a bounded domain with Lipschitz boundary and $0<s< 1\leq p<\infty$. Then, the space $\mathbb{X}_0^{s,p}(\Omega)$ is defined as the closure of $C_c^\infty(\Omega)$ under the norm $\|\cdot\|_{\mathbb{X}_0^{s,p}(\Omega)}$ given by
$$\|u\|_{\mathbb{X}_0^{s,p}(\Omega)}^p=\|\nabla u\|_{L^p(\Omega)}^p+\int_{\mathbb{R}^N}\int_{\mathbb{R}^N}|u(x)-u(y)|^pK_{s,p}(x,y)dydx.$$
Observe that by \eqref{F-Kernel}, the norm $\|\cdot\|_{\mathbb{X}_0^{s,p}(\Omega)}$ is equivalent to the norm $\|\cdot\|_{\mathbb{X}_0^{s,p}(\Omega)}'$, which is defined as
$$\|u\|_{\mathbb{X}_0^{s,p}(\Omega)}'=\left(\|\nabla u\|_{L^p(\Omega)}^p+\int_{\mathbb{R}^N}\int_{\mathbb{R}^N}\frac{|u(x)-u(y)|^p}{|x-y|^{N+sp}}dydx\right)^\frac{1}{p}.$$
 We have the following theorem.
\begin{theorem}\label{F-RSB}(See \cite[Theorem 7]{LG2025})
    Let $0<s<1$. The space $\mathbb{X}_0^{s,p}(\Omega)$ is a seperable Banach space for $1\leq p<\infty$. Also, $\mathbb{X}_0^{s,p}(\Omega)$ is reflexive for $1<p<\infty$.
\end{theorem}
As a consequence of Theorem \ref{F-SFE} and \cite[Corollary 9.19]{B2011}, we have
\begin{equation}\label{F-SS}
    \mathbb{X}_0^{s,p}(\Omega) = \{ u:\mathbb{R}^N \rightarrow\mathbb{R} \vert u\restriction_\Omega \in W_0^{1,p}(\Omega), \ u =0 \text{ in } \mathbb{R}^N \setminus \Omega \}.
\end{equation}
and the norm $\|\cdot\|_{s,p,\Omega}=\|\nabla (\cdot)\|_{L^p(\Omega)}$ is a norm equivalent to $\|\cdot\|_{\mathbb{X}_0^{s,p}(\Omega)}$ on $\mathbb{X}_0^{s,p}(\Omega)$. 
One can find in pp. 290 of \cite{B2011} that the compact embeddings of $W^{1,p}(\Omega)$ given in \cite[Theorem 9.16]{B2011} holds true in $W_0^{1,p}(\Omega)$ for any bounded domain $\Omega$. Therefore, from \eqref{F-SS} and the embeddings in \cite[Corollary 9.14, Theorem 9.16]{B2011}, the next theorem follows.
\begin{theorem}\label{F-ET}
    Let $\Omega \subset \mathbb{R}^N$ be a bounded domain with Lipschitz boundary and $0<s<1\leq p<\infty$. Then, the following continuous and compact embeddings hold.
    \begin{itemize}
        \item For $p<N$, the embedding $\mathbb{X}_0^{s,p}(\Omega)\hookrightarrow L^r(\Omega)$ is continuous for all $1\leq r\leq p^*=\frac{Np}{N-p}$ and compact for $1\leq p< p^*$.
        \item For $p=N$, we have the embedding $\mathbb{X}_0^{s,p}(\Omega)\hookrightarrow L^r(\Omega)$ which is both continuous and compact for all $p\leq r<\infty$.
        \item For $p>N$, the space $\mathbb{X}_0^{s,p}(\Omega)$ is continuously embedded in $L^r(\Omega)$ for all $1\leq r\leq \infty$. Also, the embeddings $\mathbb{X}_0^{s,p}(\Omega)\hookrightarrow L^r(\Omega)$ and $\mathbb{X}_0^{s,p}(\Omega) \hookrightarrow C(\overline{\Omega})$ are compact for all $1\leq r\leq\infty$.
    \end{itemize}
\end{theorem}
\begin{remark}\label{F-R2}
Note that by the choice of $q$, we have $p<\frac{pq}{q-1}<p^*$ for $p<N$. Therefore, it follows from Theorem \ref{F-ET} that $X_0^{(s,p)}(\Omega)$ is continuously and compactly embedded in $L^\frac{pq}{q-1}(\Omega)$ for $0<s<1<p<\infty$.
\end{remark}
\begin{notation}
We now fix notations that will be used throughout the paper.
    \begin{itemize}
        \item Let $A \subset \mathbb{R}^N$. We denote by $|A|$, the Lebesgue measure of $A$.
        \item For a function $u$ from any domain in $\Omega $ to $\mathbb{R}$, we define $u_{\pm}=\max\{\pm u,0\}$.
        \item Let $u,v:\mathbb{R}^N \rightarrow \mathbb{R}$. We denote
        \begin{align*}
              H_{s,p}(u,v):&=\int_\Omega |\nabla u|^{p-2}\nabla u \cdot \nabla v dx+\int_{\mathbb{R}^N}\int_{\mathbb{R}^N} |u(x)-u(y)|^{p-2}\\
              & \hspace{1cm} \times(u(x)-u(y))(v(x)-v(y))K_{s,p}(x,y)dydx.
     \end{align*}
    \end{itemize}
\end{notation}
 
To proceed to the next definition, we need the following definition of \textit{`Tail spaces'} (see \cite{DKP2016}).
$$L_{sp}^{p-1}(\mathbb{R}^N)=\{w\in L^{p-1}_{\text{loc}}(\mathbb{R}^N) : \int_{\mathbb{R}^N}\frac{|w(x)|^{p-1}}{(1+|x|)^{N+sp}}dx<\infty \}.$$
For each $w\in L_{sp}^{p-1}(\mathbb{R}^N)$, the tail of $w$ relative to a ball $B_r(x_0)$ is given by
$$\operatorname{Tail}(w;x_0,r)=\Big(r^{sp}\int_{\mathbb{R}^N \setminus B_r(x_0)}\frac{|w(x)|^{p-1}}{|x-x_0|^{N+sp}}dx\Big)^{\frac{1}{p-1}}.$$
Clearly, tail of the functions in $L_{sp}^{p-1}(\mathbb{R}^N)$ is finite over any ball in $\mathbb{R}^N$. Now, we discuss the notion of eigenvalues and eigenfunctions to \eqref{F}.\begin{definition}\label{F-WF}
    We say that $\lambda\in \mathbb{R}$ is an eigenvalue and the function $u\in \mathbb{X}_0^{s,p}(\Omega)\cap L_{s,p}^{p-1}(\mathbb{R}^N)$ is an eigenfunction associated with $\lambda$ for \eqref{F} with potential $V$ if the following holds.
    \begin{equation*}
        H_{s,p}(u,v)+\int_\Omega V|u|^{p-2}uvdx=\lambda\int_\Omega |u|^{p-2}uv dx, \text{ for all } v\in \mathbb{X}_0^{s,p}(\Omega).
    \end{equation*}
\end{definition}
For each $V\in L^q(\Omega)$, let us denote by $\lambda_1(V)$, the first eigenvalue of the problem \eqref{F} with potential $V$, (see Section \ref{F-s2} for the definition of $\lambda_1(\cdot)$). 
\begin{definition}\label{F-RQ}
    We define a functional $R_Q:L^q(\Omega)\rightarrow \mathbb{R}$ as
\begin{align}\label{F-MRQ}
    R_Q(V)&=\inf_{u\in \mathbb{X}_0^{s,p}(\Omega)}\frac{\|u\|_{ \mathbb{X}_0^{s,p}(\Omega)}^p+\int_\Omega V|u|^p dx}{\|u\|_{L^p(\Omega)}^p}\nonumber\\
    &=\inf\left\{\|u\|_{\mathbb{X}_0^{s,p}(\Omega)}^p+\int_\Omega V|u|^p dx:u\in \mathbb{X}_0^{s,p}(\Omega),\ \|u\|_{L^p(\Omega)}=1\right\}\nonumber\\
    &=\inf\left\{R_Q(V,u):u\in \mathbb{X}_0^{s,p}(\Omega),\ \|u\|_{L^p(\Omega)}=1\right\},
\end{align}
where $R_Q(V,u)$ denotes the Rayleigh quotient associated with $u$ for problem \eqref{F} with potential $V$. 
\end{definition}
\begin{remark}
    Thanks to Lemma \ref{F-L2}, we have Definition \ref{F-RQ} is well-defined.
\end{remark}

\section{The eigenvalue problem}\label{F-s2}
\noindent In this section, we discuss the existence of the first eigenvalue and characterize the eigenvalues and the associated eigenfunctions of \eqref{F} for a potential $V\in L^q(\Omega)$. The results follow from arguments similar to \cite{LGG2026, PBR2018} and from the embedding results in Theorem \ref{F-ET}. Nevertheless, we provide the outline of some proofs to make the paper self-contained. We begin with a result on the boundedness of eigenfunctions.
\begin{theorem}\label{F-E.T2}
    Let $(\lambda,u)$ be an eigenpair of \eqref{F} with potential $V$. Then, $u$ is bounded in $\mathbb{R}^N$.  
\end{theorem}
\begin{proof}
    Since $u\in \mathbb{X}_0^{s,p}(\Omega)$, we have $u \equiv 0 $ in $\mathbb{R}^N \setminus\Omega$. Thus, it suffices to prove that $u$ is bounded in $\
    \Omega$. Observe that for case $p>N$, the theorem holds by the continuous embedding $\mathbb{X}_0^{s,p}(\Omega)\hookrightarrow L^\infty(\Omega)$ from Theorem \ref{F-ET}. Now, consider the case $p \leq N$. Let $u_i:=(u-(1-\frac{1}{2^i}))_+$ for $i\in\mathbb{N}$. Let us define the sequence $(b_i)$ by
    $$b_i:=\|u_i\|_{L^{\frac{pq}{q-1}}(\Omega)}^p \text{ for all } i\in \mathbb{N} \cup\{0\}.$$
    Observe that $u_0=u_+$ and $u_i \rightarrow (u-1)_+$ pointwise as $i \rightarrow \infty$. Since we have $|u_i|\leq |u|+1 \in L^{\frac{pq}{q-1}}(\Omega)$ for all $i\in \mathbb{N}$ by Remark \ref{F-R2}, we obtain using the dominated convergence theorem that $b_i \rightarrow \|(u-1)_+\|_{L^{\frac{pq}{q-1}}(\Omega)}$ as $i \rightarrow \infty$. Consider an $i \in \mathbb{N} \cup\{0\}$.  Fix $\eta>\frac{pq}{q-1}$. Let
    $$\theta=\begin{cases}
        p^*, \ &\text{ if } p<N,\\
        \eta, \ &\text{ if } p=N.
    \end{cases}$$
    It is easy to see that $\{u_{i+1}>0\}\subset \{u_i>\frac{1}{2^{i+1}}\}$ for all $i\in \mathbb{N} \cup\{0\}$. Also, using the H\"older inequality, we have
    \begin{align*}
       \left| \left\{u_i>\frac{1}{2^{i+1}}\right\}\right|&=\int_{\left\{u_i>\frac{1}{2^{i+1}}\right\}}1 dx\\
       &\leq \left(\int_{\left\{u_i>\frac{1}{2^{i+1}}\right\}}u_i^\frac{pq}{q-1} dx\right)^\frac{q-1}{pq}  \left(\int_{\left\{u_i>\frac{1}{2^{i+1}}\right\}}\left(\frac{1}{u_{i+1}}\right)^{\frac{pq}{pq-(q-1)}}dx\right)^{1-\frac{q-1}{pq}}\\
       &\leq \|u_i\|_{L^\frac{pq}{q-1}(\Omega)}2^{(i+1)}\left|\left\{u_i>\frac{1}{2^{i+1}}\right\}\right|^{1-\frac{q-1}{pq}}.
    \end{align*}
    Therefore, we have
    $$\left|\{u_{i+1}>0\}\right|\leq \|u_i\|_{L^\frac{pq}{q-1}(\Omega)}^\frac{pq}{q-1}2^{(i+1)\frac{pq}{q-1}}.$$
    Thus, applying the H\"older inequality with exponents $\frac{\theta(q-1)}{pq}$ and $\frac{\theta(q-1)}{\theta(q-1)-pq}$ and using the continuous embedding given by Theorem \ref{F-ET}, we deduce that
    \begin{align}\label{F-E.T2-2}
        \|u_{i+1}\|_{L^\frac{pq}{q-1}(\Omega)}^p&=\left(\int_{\{u_{i+1}>0\}}u_{i+1}^\frac{pq}{q-1}dx\right)^{\frac{q-1}{q}}\nonumber\\
        &\leq \|u_{i+1}\|_{L^\theta(\Omega)}^p|\{u_{i+1}>0\}|^{\frac{q-1}{q}-\frac{p}{\theta}}\nonumber\\
        &\leq C\|u_{i+1}\|_{\mathbb{X}_0^{s,p}(\Omega)}^p\left(\|u_i\|_{L^\frac{pq}{q-1}(\Omega)}2^{i+1}\right)^{\frac{pq}{q-1}\left(\frac{q-1}{q}-\frac{p}{\theta}\right)}\nonumber\\
        &\leq C\|u_{i+1}\|_{\mathbb{X}_0^{s,p}(\Omega)}^p\left(\|u_i\|_{L^\frac{pq}{q-1}(\Omega)}^p2^{ip}\right)^{1-\frac{pq}{\theta(q-1)}}.
    \end{align}
    Observe that $u<(2^{i+1}-1)u_i$ by the definition of $u_i$. Following the derivation of equations (5.42) and (5.43) in \cite[Theorem 5.3]{LGG2026} and substituting them in the weak formulation of \eqref{F} with the test function $u_i$ (see Definition \ref{F-WF}), we deduce the following.
    \begin{align}\label{F-E.T2-1}
        \|u_i\|_{\mathbb{X}_0^{s,p}(\Omega)}^p&\leq H_{s,p}(u,u_i)\nonumber\\
        &=\int_\Omega (\lambda-V)|u|^{p-2}u u_i dx \nonumber\\
        &\leq (2^{i+1}-1)^{p-1}\int_\Omega (|\lambda|+|V|)|u_i|^{p}dx\nonumber\\
        &\leq (2^{i+1}-1)^{p-1}\left(|\lambda|\cdot|\Omega|^\frac{1}{q}+\|V\|_{L^q(\Omega)}\right)\|u_i\|_{L^{\frac{pq}{q-1}}(\Omega)}^p\nonumber\\
        &\leq C2^{ip}\|u_i\|_{L^{\frac{pq}{q-1}}(\Omega)}^p. 
    \end{align}
    The penultimate step in \eqref{F-E.T2-1} is obtained using the H\"older inequality. Combining \eqref{F-E.T2-2} and \eqref{F-E.T2-1}, we deduce that
    \begin{equation}\label{F-E.T2-3}
        b_{i+1}\leq C\left(2^{ip}b_i\right)^{1+\alpha},
    \end{equation}
    where $\alpha:=1-\frac{pq}{\theta(q-1)}>0$ by the choice of $\theta$. Now, since $cu$ is an eigenfunction associated with $\lambda$ of \eqref{F} for all $c\in \mathbb{R}\setminus\{0\}$, choose $u$ such that $b_0\leq C^\frac{-1}{\alpha^2}$. Then, \eqref{F-E.T2-3} gives
    $$\|(u-1)_+\|_{L^\frac{pq}{q-1}(\Omega)}^p=\lim\limits_{i \rightarrow \infty}b_i=0.$$
    Thus, $u_+$ is bounded in $\Omega$. Replacing $u$ by $-u$ in the above proof, we obtain that $u_-$ is bounded in $\Omega$. This completes the proof.
\end{proof}
Next, we give a unique characterization of the first eigenvalue $\lambda_1(V)$ of \eqref{F}.
\begin{theorem}\label{F-E.T1}
    For each $V\in L^q(\Omega)$, the minimum of Rayleigh quotients, $R_Q(V)$ defined by \eqref{F-MRQ} is attained by a non-zero $u$ in $ \mathbb{X}_0^{s,p}(\Omega)$ with $\|u\|_{L^p(\Omega)}=1$. Also, $R_Q(V)=\lambda_1(V)$ and $u$ is an eigenfunction associated with $\lambda_1(V)$ for \eqref{F} with potential $V$. 
\end{theorem}
\begin{proof}
    If $R_Q(V)$ is attained at $ u \in \mathbb{X}_0^{s,p}(\Omega)$ having $\|u\|_{L^p(\Omega)}=1$, clearly $R_Q(V)$ is an eigenvalue and $u$ is an associated eigenfunction. Now, let $(\lambda, \tilde{u})$ be an eigenpair of \eqref{F} with $\|\tilde{u}\|_{L^p(\Omega)}=1$. Since $\tilde{u}$ can be taken as a test function in the weak formulation of \eqref{F} given by Definition \ref{F-WF}, we deduce that
    $$\lambda=H_{s,p}(\tilde{u},\tilde{u})+\int_{\Omega}V|\tilde{u}|^p dx=\|\tilde{u}\|_{ \mathbb{X}_0^{s,p}(\Omega)}^p+\int_{\Omega}V|\tilde{u}|^p dx\geq R_Q(V).$$
    Thus, $\lambda_1(V)$ is the first eigenvalue of \eqref{F}. It only remains to prove that $R_Q(V)$ is attained in $\mathbb{X}_0^{s,p}(\Omega)$. To prove this, consider a minimizing sequence $(u_i) \subset \{u \in \mathbb{X}_0^{s,p}(\Omega): \|u\|_{L^p(\Omega)}=1\}$ of $R_Q(V)$. Proceeding similarly to the proof of \cite[Theorem 4.1]{LGG2026}, using Theorem \ref{F-RSB} and the compact embedding of $\mathbb{X}_0^{s,p}(\Omega)$ from Theorem \ref{F-ET}, we deduce that up to a subsequence, $(u_i)$ weakly converges to a function $u$ in $\mathbb{X}_0^{s,p}(\Omega)$ with $\|u\|_{L^p(\Omega)}=1$ (thus $u \not \equiv 0$) and 
    $$R_Q(V)=\|u\|_{\mathbb{X}_0^{s,p}(\Omega)}^p+\int_\Omega V|u|^p dx.$$
    Hence, the proof is complete.
\end{proof}
We now prove a comparison principle, which is required to establish the strict positivity of eigenfunctions associated with $\lambda_1(V)$ of \eqref{F}.
\begin{lemma}\label{F-E.L1}[Comparison Principle]
    Let $V\in L^q(\Omega)$ and $u \not \equiv 0$ be an element of $\mathbb{X}_0^{s,p}(\Omega)$ that satisfies
    \begin{equation}\label{F-E.L1-1'}
        H_{s,p}(u,v)+\int_\Omega V|u|^{p-2}uvdx\geq 0, \text{ for all } v\in \mathbb{X}_0^{s,p}(\Omega).
    \end{equation}
    Then, $u>0$ a.e. in $\Omega$.
\end{lemma}
\begin{proof}
    If possible, assume that there is a connected component $\Omega'$ of $\Omega$ such that $u \equiv 0$ in $\Omega'$. Let $v\in C_c^\infty(\Omega')$ with $v\geq 0$. On taking $v$ as a test function in \eqref{F-E.L1-1'}, we deduce
    $$\int_{\mathbb{R}^N \setminus \Omega'}\phi(y)\left(\int_{\Omega'}u(x)^{p-1}K_{s,p}(x,y)dx\right)dy=0.$$
    Since $\phi \in C_c^\infty(\Omega')$ is an arbitrary non-negative function, we arrive at
    $$\int_{\Omega'}u(x)^{p-1}K_{s,p}(x,y)dx=0 \text{ for all } y \in \mathbb{R}^N.$$
    However, from \eqref{F-Kernel}, for each $x\in \mathbb{R}^N$, we have $K_{s,p}(x,y)>0$ for a.e. $y \in \mathbb{R}^N$. Combining this result with the non-negativity of $u$, it follows that $u \equiv 0$ in $\mathbb{R}^N \setminus \Omega'$. Thus, $u \equiv 0$, which contradicts our assumption. Hence, we have $u \not \equiv 0$ in all connected components of $\Omega$. 
    For $\epsilon>0$, define 
    $$T_\epsilon(x)=\log\left(1+\frac{u(x)}{\epsilon}\right).$$
    Let $A=\{x\in \Omega: u(x)=0\}$. If possible, assume that $|A|>0$. Then, consider a ball $B_{2r}(x_0)\subset \Omega$ such that $u \not \equiv 0$ in $B:=B_r(x_0)$ and $|B\cap A|>0$. Note that for any $\psi\in C_c^\infty(B_{\frac{3r}{2}}(x_0)$ with $0\leq \psi \leq 1$, and $\psi \equiv 1$ in $B$, the embedding in Theorem \ref{F-ET} along with the H\"older inequality gives
    \begin{equation}\label{F-E.T3-1'}
        \int_{B_{2r}(x_0)} Vu^{p-1}(u+\epsilon)^{1-p}\psi dx\leq Cr^{\frac{q-1}{pq}}\|V\|_{L^q(\Omega)}.
    \end{equation}
     Following the steps in the derivation of \cite[Lemma 3.2]{LGG2026} along with \eqref{F-E.T3-1'}, we arrive at
    \begin{equation}\label{F-E.T3-1}
        \int_B |\nabla (u(x)+\epsilon)|^p dx +\int_B\int_B\left|\log\left(\frac{u(x)+\epsilon}{u(y)+\epsilon}\right)\right|^p\frac{1}{|x-y|^{N+sp}}dydx\leq C,
    \end{equation}
    where $C=C(N,s,p,r,u,V)>0$ is a constant. Now, for $x\in B$ and $y \in B\cap A$, we have
    \begin{equation}\label{F-E.T3-1''}
        |T_\epsilon(x)|^p=|T_\epsilon(x)-T_\epsilon(y)|^p\leq \frac{2r}{|x-y|^{N+sp}}\left|\log\left(\frac{u(x)+\epsilon}{u(y)+\epsilon}\right)\right|^p.
    \end{equation}
    First taking the average in \eqref{F-E.T3-1''} over $y\in B\cap A$, then integrating over $x\in B$ and finally using \eqref{F-E.T3-1}, we have
    \begin{align*}
        \int_B|T_\epsilon(x)|^p dx &\leq\frac{(2r)^{N+sp}}{|B \cap A|}\int_B\int_B\left|\log\left(\frac{u(x)+\epsilon}{u(y)+\epsilon}\right)\right|^p\frac{1}{|x-y|^{N+sp}}dydx\leq \frac{C}{|B \cap A|}.
    \end{align*}
   Taking $\epsilon\rightarrow 0$, we infer that $u \equiv 0$ a.e. in $B$. This contradicts our initial assumption. Hence, we deduce $u>0$ a.e. in $\Omega$, which is the desired result.
\end{proof}

We now establish some properties of the sign of eigenfunctions to \eqref{F}, depending on the corresponding eigenvalue.
\begin{theorem}\label{F-E.T3}
    Let $u$ be a non-zero eigenfunction associated with $\lambda_1(V)$ for \eqref{F}. Then, either $u>0$ a.e. in $\Omega$ or $u<0$ a.e. in $\Omega$.
\end{theorem}
\begin{proof}
    If $u$ changes sign in $\Omega$ in a subset of $\Omega$ with positive measure, we get
    $$\||u|\|_{\mathbb{X}_0^{s,p}(\Omega)}^p+\int_\Omega V|u|^p dx <\|u\|_{\mathbb{X}_0^{s,p}(\Omega)}^p+\int_\Omega V|u|^p dx=\lambda_1(V),$$
    which contradicts Theorem \ref{F-E.T1}. Consequently, $u$ does not change sign a.e. in $\Omega$. Without loss of generality, assume that $u \geq 0$ a.e. in $\Omega$. BY the weak formulation \eqref{F-WF}, we have
    $$H_{s,p}(u,v)+\int_\Omega (V-\lambda_1(V))|u|^{p-2}uvdx= 0, \text{ for all } v\in \mathbb{X}_0^{s,p}(\Omega).$$
    Also, $V-\lambda_1(V)\in L^q(\Omega)$. Thus, by Lemma \ref{F-E.L1}, we have $u>0$ a.e. in $\Omega$. If $u$ is a non-positive eigenfunction, we get the required result by replacing $u$ with $-u$ in the previous argument. 
\end{proof}

\begin{theorem}\label{F-E.T3'}
    Let $v$ be a non-zero eigenfunction associated with $\lambda>\lambda_1(V)$ for \eqref{F}. Then, $v$ changes sign in $\Omega$.
\end{theorem}
\begin{proof}
    We proceed by the method of contradiction. If possible, assume that $v$ does not change sign in $\Omega$. Without loss of generality, let $v \geq 0$. By Theorem \ref{F-E.T3}, we have $v>0$. Let $\psi_i$ be defined by 
    $$\psi_i=\frac{u^p}{\left(v+\frac{1}{i}\right)^{p-1}}, \text{ for all }i\in \mathbb{N}.$$
    Obviously,  ${\left(v+\frac{1}{i}\right)^{p-1}}\geq \frac{1}{i^{p-1}}$. Since we have $\Omega$ is a bounded domain, and using the fact that $u\in L^\infty(\Omega)$ (by Theorem \ref{F-E.T2}), we have $\psi_i\in L^p(\Omega)$. Similarly, we get
    \begin{align*}
        |\nabla \psi_i|&=\left|\frac{p\left(v+\frac{1}{i}\right)^{p-1}u^{p-1}\nabla u-(p-1)u^p\left(v+\frac{1}{i}\right)^{p-2}\nabla v}{\left(v+\frac{1}{i}\right)^{2(p-1)}}\right|\\
        &\leq C\left(\nabla u+\nabla v\right)\in L^p(\Omega).
    \end{align*}
    Here $C=C\left(i,p,\|u\|_{L^{\infty}(\Omega)}\right)>0$ is a constant. Also, $\psi_i=0 \in \mathbb{R}^N \setminus \Omega$. Therefore, it follows that $\psi_i\in \mathbb{X}_0^{s,p}(\Omega)$. Taking $\psi_i$ as a test function in the weak formulation in Definition \ref{F-WF} and applying Picone's identity (see \cite[Theorem 1.1]{AH1998}, \cite[Lemma 6.2]{A2008} and \cite[Lemma 3.5]{PBR2018}), we get
    \begin{align}\label{F-E.T3'-1}
        0&\leq H_{s,p}(u,u)-H_{s,p}(v,\psi_i)\nonumber\\
        &= \int_\Omega\left(\lambda_1(V)-V\right) |u|^p dx-\int_\Omega\left(\lambda-V\right)v^{p-1}\psi_idx.
    \end{align}
    Observe that $v^{p-1}\psi_i \leq u^p$ for all $i \in \mathbb{N}$. Since $u\in L^p(\Omega)\cap L^{\frac{pq}{q-1}}(\Omega)$, using  the H\"older inequality, we deduce that $\left(\lambda-V\right)u^p \in L^1(\Omega)$. Also, note that
    $$\left(\lambda-V\right)v^{p-1}\psi_i \rightarrow \left(\lambda-V\right)u^p \text{ a.e. in }\Omega.$$
    Recall that we have $\lambda>\lambda_1$. Hence, applying dominated convergence theorem in \eqref{F-E.T3'-1} yields
    $$H_{s,p}(u,u)-H_{s,p}(v,\psi_i)=0.$$
    Again from Picones's identity, 
    we obtain $u=cv$ for some constant $c$, which contradicts $\lambda\neq \lambda_1$. This proves the theorem. 
\end{proof}
We end this section with the following two results on the first eigenvalue to \eqref{F}. 
\begin{theorem}\label{F-E.T4}
    The first eigenvalue $\lambda_1(V)$ of \eqref{F} is simple.
\end{theorem}
\begin{proof}
    Let $u,v>0$ be strictly positive eigenfunctions associated with $\lambda_1$ for \eqref{F} with $\|u\|_{L^p(\Omega)}=\|u\|_{L^p(\Omega)}=1$. Let us take
    $$w=\left(\frac{u^p+v^p}{2}\right)^{\frac{1}{p}}.$$
    Using Theorem \ref{F-E.T1} and proceeding as in the proof of (4.12) in the proof of \cite[Theorem 4.5]{LGG2026}, we obtain
    \begin{align}\label{F-E.T4-1}
        \lambda_1(V)&\leq H_{s,p}(w,w)+\int_\Omega V|w|^p dx \nonumber\\
        &\leq \frac{1}{2}\left(H_{s,p}(u,u)+\int_\Omega V|u|^p dx+H_{s,p}(v,v)+\int_\Omega V|v|^p dx\right)\nonumber\\
        &=\lambda_1(V).
    \end{align}
    Proceeding as in the proof of \cite[Theorem 4.5]{LGG2026} using the strict convexity of the map $t \mapsto |t|^p$, we get that equality holds in the penultimate step of \eqref{F-E.T4-1} only if $\left|\nabla \left(\frac{u}{v}\right)\right|=0$. i.e., there exists a constant $c$ such that $u=cv$. This proves the desired result.
\end{proof}

\begin{theorem}\label{F-E.T5}
    The first eigenvalue $\lambda_1(V)$ of \eqref{F} is isolated.
\end{theorem}
\begin{proof}
    Let $(\lambda,v)$ be an eigenpair of \eqref{F} with $\lambda>\lambda_1$. Then, $v_-\in \mathbb{X}_0^{s,p}(\Omega)$. Taking $v_-$ as a test function in the weak formulation given by Definition \ref{F-WF} and following the proof of (4.33) in \cite[Theorem 4.5]{LGG2026}, we have 
    \begin{align}\label{F-E.T5-1}
        \lambda\int_\Omega |v_-|^p dx&=-H_{s,p}(v,v_-)+\int_\Omega V|v_-|^p dx \nonumber\\
        &\geq \|v_-\|_{\mathbb{X}_0^{s,p}(\Omega)}^p+\int_\Omega V|v_-|^p dx.
    \end{align}
     Fix a real number $\eta>\frac{pq}{q-1}$ such that $\eta<p^*$ when $p<N$.
     Clearly, $\eta>p$ and $\frac{\eta}{\eta-p}<q$. Using the continuous embedding of $\mathbb{X}_0^{s,p}(\Omega)$ in $L^\eta(\Omega)$ from Theorem \ref{F-ET} and then applying the H\"older inequality twice in \eqref{F-E.T5-1}, we deduce 
     \begin{align}\label{F-E.T5-2}
         \|v_-\|_{L^\eta(\Omega)}^p&\leq C\|v_-\|_{\mathbb{X}_0^{s,p}(\Omega)}^p\nonumber\\
         &\leq C\int_{\Omega \cap \{v<0\}} \left(\lambda-V\right)|v_-|^p dx \nonumber\\
         &\leq C \left(\int_{\Omega \cap \{v<0\}}(|\lambda|+|V|)^{\frac{\eta}{\eta-p}}dx\right)^\frac{\eta-p}{\eta}\|u\|_{L^\eta(\Omega)}^p \nonumber\\
         &\leq C \left(|\lambda|\cdot|\Omega \cap \{v<0\}|^{\frac{\eta-p}{\eta}}+\|V\|_{L^q(\Omega)}^{\frac{\eta-p}{\eta}}\cdot|\Omega \cap \{v<0\}|^{1-\frac{\eta}{q(\eta-p)}}\right)\|u\|_{L^\eta(\Omega)}^p.
     \end{align}
     Using the fact that $\|v_-\|_{L^\eta(\Omega)}>0$ (by Theorem \ref{F-E.T3'}) in \eqref{F-E.T5-2}, it follows that
     \begin{align*}\label{F-E.T5-3}
         1&\leq C \left(|\lambda|\cdot|\Omega \cap \{v<0\}|^{\frac{\eta-p}{\eta}}+\|V\|_{L^q(\Omega)}^{\frac{\eta-p}{\eta}}\cdot|\Omega \cap \{v<0\}|^{1-\frac{q(\eta-p)}{\eta}}\right) \\
         &\leq 2C \max \left\{|\lambda|\cdot|\Omega \cap \{v<0\}|^{\frac{\eta-p}{\eta}},\|V\|_{L^q(\Omega)}^{\frac{\eta-p}{\eta}}\cdot|\Omega \cap \{v<0\}|^{1-\frac{q(\eta-p)}{\eta}}\right\}.
     \end{align*}
     Thus, we get
     \begin{equation}\label{F-E.T5-4}
         |\Omega \cap \{v<0\}| \geq C\min\left\{(|\lambda|+1)^{\frac{-\eta}{\eta-p}},\|V\|_{L^q(\Omega)}^{\frac{-(\eta-p)}{\eta-q(\eta-p)}} \right\}=C(\lambda),
     \end{equation}
     where $C(\lambda)>0$ decreases as $|\lambda|$ increases. Replacing $v$ by $-v$ in the previous steps, we obtain \eqref{F-E.T5-4} with $\Omega \cap \{v<0\}$ replaced by $\Omega \cap \{v>0\}$. The rest of the proof follows by the method of contradiction. If possible, assume that $\lambda_1(V)$ is not isolated. Consider a sequence of eigenpairs $(\mu_i,u_i)$ satisfying $\mu_i>\lambda_1(V)$ for all $i\in \mathbb{N}$ with $\mu_i \searrow \lambda_1(V)$ and $\|u_i\|_{L^p(\Omega)=1}$ for all $i\in \mathbb{N}$. Since $(\mu_i)$ is a bounded sequence, we choose $\mu > 0$ such that
     $$|\mu\geq |\mu_i| \text{ for all }i \in \mathbb{N}.$$
     Following steps similar to \cite[Theorem 4.5]{LGG2026}, we get a weak limit $u\in \mathbb{X}_0^{s,p}(\Omega)$ of a subsequence of $(u_i)$ (still denoted by $(u_i)$) with $\|u\|_{L^p(\Omega)}=1$, such that $u$ is an eigenfunction associated with $\lambda_1(V)$ for \eqref{F}. Without loss of generality, we may assume $u>0$ by Theorem \ref{F-E.T3}. However, we have
     \begin{equation}\label{F-E.T5-4'}
         |\Omega \cap \{\mu_i<0\}| \geq C(\mu) \text{ for all }i \in \mathbb{N},
     \end{equation}
     by \eqref{F-E.T5-4}. Let $d=\frac{1}{4}C(\mu)$. By Egorov's theorem, there exists a compact set $\Omega'\subset \Omega$ such that $|\Omega \setminus \Omega'|<d$ and $u_i \rightarrow u$ uniformly in $\Omega'$. By the compactness of $\Omega'$, we also obtain $\alpha>0$ such that $u>\alpha$ in $\Omega'$. Then, the uniform convergence of $(u_i)$ to $u$ in $\Omega'$ gives a $j\in \mathbb{N}$ that satisfies $|u_j-u|<\frac{\alpha}{2}$. Then, $u_j>0$ in $\Omega'$, which contradicts \eqref{F-E.T5-4'}. This establishes the proof. 
\end{proof}

\section{Optimization in a bounded, closed convex set}\label{F-s3}
\noindent In this section, we prove the existence of functions $V$ that optimize the eigenvalue $\lambda_1(V)$ for \eqref{F} on bounded, closed and convex subsets of $L^q(\Omega)$ and discuss some properties of these optimal functions. To establish the main theorems, we need to discuss some properties of the functional $\lambda_1(\cdot)$. We begin with the following lemma.
\begin{lemma}\label{F-L1}
    The functional $\lambda_1(\cdot)$ is concave.
\end{lemma}
\begin{proof}
    Let $V_1,V_2\in L^q(\Omega)$ and $0\leq t\leq 1$. Clearly, $tV_1+(1-t)V_2\in L^q(\Omega)$. Using the characterization of $\lambda_1(\cdot)$ in Theorem \ref{F-E.T1}, we get
    \begin{align}\label{F-L1-1}
        t\lambda_1(V_1)+(1-t)\lambda_1(V_2)&\leq t\left(\|u\|_{\mathbb{X}_0^{s,p}(\Omega)}^p+\int_\Omega V_1|u|^p dx\right)+\left(\|u\|_{\mathbb{X}_0^{s,p}(\Omega)}^p+\int_\Omega V_2|u|^p dx\right)\nonumber\\
        &=\|u\|_{\mathbb{X}_0^{s,p}(\Omega)}^p+\int_\Omega \left(tV_1+(1-t)V_2\right)|u|^p dx,
    \end{align}
    for all $u\in \mathbb{X}_0^{s,p}(\Omega)$ with $\|u\|_{L^p(\Omega)}=1$. Taking the infimum in \eqref{F-L1-1} over the set $\{u\in \mathbb{X}_0^{s,p}(\Omega):\|u\|_{L^p(\Omega)}=1\}$, we deduce  
    $$t\lambda_1(V_1)+(1-t)\lambda_1(V_2)\leq \lambda_1\left(tV_1+(1-t)V_2\right),$$
    which proves the desired result.
\end{proof}

\begin{lemma}\label{F-L2}
    Let $\mathcal{A}\subset L^q(\Omega)$ and $M>0$ be a constant such that $\|V\|_{L^q(\Omega)}\leq M$ for every $V\in \mathcal{A}$. Then, there exists a constant $C>0$ such that $\lambda_1(V)\leq C$ for all $V\in \mathcal{A}$. Also, for each $\delta>0$, there exists a constant $C(\delta)>0$ such that 
    $$\left|\int_\Omega V|u|^p dx\right|\leq \delta\|u\|_{\mathbb{X}_0^{s,p}(\Omega)}^p+C(\delta)\|V\|_{L^q(\Omega)}\|u\|_{L^p(\Omega)}^p,$$
    for all $V\in \mathcal{A}$ and $u\in \mathbb{X}_0^{s,p}(\Omega)$.
\end{lemma}
\begin{proof}
    Let $V\in \mathcal{A}$. Fix a function $w\in \mathbb{X}_0^{s,p}(\Omega)$ that satisfies $\|w\|_{L^p(\Omega)}=1$. From Remark \ref{F-R2}, it follows  that $\|w\|_{L^{\frac{pq}{q-1}}(\Omega)}<\infty$. Then, by Theorem \ref{F-E.T1} and using the H\"older inequality, we have
    \begin{align*}
        \lambda_1(V)&\leq \|w\|_{\mathbb{X}_0^{s,p}(\Omega)}^p+\int_\Omega V|w|^p dx \\
        &\leq \|w\|_{\mathbb{X}_0^{s,p}(\Omega)}^p+\left(\int_\Omega |V|^q dx\right)^{\frac{1}{q}}\left(\int_\Omega|w|^{\frac{pq}{q-1}} dx\right)^{\frac{q-1}{q}}\\
        &\leq \|w\|_{\mathbb{X}_0^{s,p}(\Omega)}^p+M\|w\|_{L^{\frac{pq}{q-1}}(\Omega)}^p=C.
    \end{align*}
    This completes the proof of the first part of Lemma \ref{F-L2}. \\
    Next, fora contradictory, let us assume that there exist $\tilde{\delta}>0$, and sequences $(V_i)\subset \mathcal{A}$ and $(u_i)\subset \mathbb{X}_0^{s,p}(\Omega)$ for $i\in \mathbb{N}$ that satisfy
    \begin{equation}\label{F-L2-2}
        \left|\int_\Omega V_i|u_i|^p dx\right|\geq \tilde{\delta}\|u_i\|_{\mathbb{X}_0^{s,p}(\Omega)}^p+i\|V_i\|_{L^q(\Omega)}\|u_i\|_{L^p(\Omega)}^p, \text{ for all }i\in \mathbb{N}.
    \end{equation}
    Without loss of generality, assume that $\|u_i\|_{L^{\frac{pq}{q-1}}(\Omega)}=1$ for all $i\in \mathbb{N}$. 
    Employing the H\"older inequality in LHS of \eqref{F-L2-2}, we obtain
    \begin{align}\label{F-L2-3}
        \tilde{\delta}\|u_i\|_{\mathbb{X}_0^{s,p}(\Omega)}^p+i\|V_i\|_{L^q(\Omega)}\|u_i\|_{L^p(\Omega)}^p&\leq \left|\int_\Omega V_i|u_i|^p dx\right|\nonumber\\
        &\leq \left(\int_\Omega |V|^q dx\right)^{\frac{1}{q}}\left(\int_\Omega|u_i|^{\frac{pq}{q-1}} dx\right)^{\frac{q-1}{q}}\nonumber\\
        &=\|V\|_{L^q(\Omega)}\|u_i\|_{L^{\frac{pq}{q-1}}(\Omega)}^p\nonumber\\
        &\leq M.
    \end{align}
    Hence, we deduce that $(u_i)$ is a bounded sequence in $\mathbb{X}_0^{s,p}(\Omega)$. Using the reflexivity of $\mathbb{X}_0^{s,p}(\Omega)$ obtained by Theorem \ref{F-RSB}, we see that $(u_i)$ has a weakly convergent subsequence (still denoted by $(u_i)$) in $\mathbb{X}_0^{s,p}(\Omega)$. Let $u_i \rightharpoonup u$ in $\mathbb{X}_0^{s,p}(\Omega)$. By the compact embedding in Theorem \ref{F-T2} and Remark \ref{F-R2}, we deduce that $u_i \rightarrow u$ up to a subsequence (we denote it by $(u_i)$ itself) in both $L^p(\Omega)$ and $L^{\frac{pq}{q-1}}(\Omega)$. Therefore, $\|u\|_{L^{\frac{pq}{q-1}}(\Omega)}=1$. Note that $(V_i)$ is a bounded sequence in $L^q(\Omega)$. Since $L^q(\Omega)$ is also a reflexive Banach space, it has a weakly convergent subsequence (we denote it by $(V_i)$ itself) in $L^q(\Omega)$. Let $V_i \rightharpoonup V$ in $L^q(\Omega)$. It is easy to see that 
    \begin{equation*}
        0\leq \liminf\limits_{i \rightarrow \infty}i\|V_i\|_{L^q(\Omega)}\|u_i\|_{L^p(\Omega)}^p\leq \limsup\limits_{i \rightarrow \infty}i\|V_i\|_{L^q(\Omega)}\|u_i\|_{L^p(\Omega)}^p\leq M,
    \end{equation*}
    from  \eqref{F-L2-3}. Thus, $\|V_i\|_{L^q(\Omega)}\|u_i\|_{L^p(\Omega)}^p \rightarrow 0$ as $i \rightarrow \infty$. Note that $u_i \rightarrow u \neq 0$ in $L^p(\Omega)$. Hence, using the weak lower semi-continuity of the norm $\|\cdot\|_{L^q(\Omega)}$, we get
    \begin{equation}\label{F-L2-6}
        \|V\|_{L^q(\Omega)}.\|u\|_{L^p(\Omega)}^p \leq \liminf\limits_{i \rightarrow \infty}\|V_i\|_{L^q(\Omega)}.\|u_i\|_{L^p(\Omega)}^p=0.
    \end{equation}
    Therefore $V\equiv 0$. Now, taking $i \rightarrow \infty$ in \eqref{F-L2-3} and using \eqref{F-L2-6} and using the fact that $\|u\|_{\mathbb{X}_0^{s,p}(\Omega)}$ is weak lower semicontinuous, we arrive at $\|u\|_{\mathbb{X}_0^{s,p}(\Omega)}=0$. Thus, $u = 0$ a.e. in $\Omega$, which contradicts $\|u\|_{L^\frac{pq}{q-1}(\Omega)}=1$. 
\end{proof}
We now prove Theorem \ref{F-T1}, which gives the existence and characterizes the minimizer of $\lambda_1(\cdot)$ in a closed, convex, and bounded subset of $L^q(\Omega)$.\vspace{.2cm} \\
\noindent {\it{\bf{Proof of Theorem} \ref{F-T1}.}}
    Let $L_1:=\inf\{\lambda_1(V): V\in \mathcal{A}\}$. Then, there exists a sequence $(V_i)$ in $\mathcal{A}$ such that $\lambda_1(V_i) \rightarrow L_1$ as $i \rightarrow \infty$. As a consequence of Theorem \ref{F-E.T1}, there exist eigenfunctions $u_i$ corresponding to the first eigenvalue $\lambda_1(V_i)$ of each $V_i$, respectively, such that $\|u_i\|_{L^p(\Omega)}=1$ and
    \begin{equation}\label{F-T1-1}
        \|u_i\|_{\mathbb{X}_0^{s,p}(\Omega)}^p+\int_\Omega V_i|u_i|^p dx=\lambda_1(V_i), \ i\in \mathbb{N}.
    \end{equation}
    Since $(V_i) \subset \mathcal{A}$ is bounded in the reflexive space $L^q(\Omega)$, there exists a subsequence of $(V_i)$ (which we still denote by $(V_i)$) and $\underline{V}\in L^q(\Omega)$ such that $V_i \rightharpoonup \underline{V}$ in $L^q(\Omega)$. Since $\mathcal{A}$ is closed and convex, it is weakly closed. This implies $\underline{V}\in \mathcal{A}$. Using Lemma \ref{F-L2}, we get two constants $C_1=C(\frac{1}{2})>0$ and $C_2>0$ such that    
    \begin{equation}\label{F-T1-2}
        \left|\int_\Omega V_i|u_i|^p dx\right|\leq \frac{1}{2}\|u_i\|_{\mathbb{X}_0^{s,p}(\Omega)}^p+C_1\|V_i\|_{L^q(\Omega)}.\|u_i\|_{L^p(\Omega)}^p\leq \frac{1}{2}\|u_i\|_{\mathbb{X}_0^{s,p}(\Omega)}^p+C_1M,
    \end{equation}
    and $\lambda_1(V_i) \leq C_2$ for all $i\in \mathbb{N}$. Thus, combining \eqref{F-T1-1} and \eqref{F-T1-2}, we deduce
    \begin{align}\label{F-T1-3}
        C_2&\geq \|u_i\|_{\mathbb{X}_0^{s,p}(\Omega)}^p+\int_\Omega V_i|u_i|^p dx \nonumber\\
        &\geq \|u_i\|_{\mathbb{X}_0^{s,p}(\Omega)}^p-\left(\frac{1}{2}\|u_i\|_{\mathbb{X}_0^{s,p}(\Omega)}^p+C_1M\right) \nonumber\\
        &=\frac{1}{2}\|u_i\|_{\mathbb{X}_0^{s,p}(\Omega)}^p-C_1M.
    \end{align}
    From \eqref{F-T1-3}, we deduce that $(\|u_i\|_{\mathbb{X}_0^{s,p}(\Omega)}^p)$ is bounded by the constant $2(C_2+C_1M)$. Hence $(u_i)$ is bounded in $\mathbb{X}_0^{s,p}(\Omega)$. Since we have $\mathbb{X}_0^{s,p}(\Omega)$ is a reflexive Banach space from Theorem \ref{F-RSB}, $u_i \rightharpoonup \underline{u}$ in $\mathbb{X}_0^{s,p}(\Omega)$ up to a subsequence (still denoted by $(u_i)$). By the compact embedding in Theorem \ref{F-ET} and Remark \ref{F-R2}, we have $u_i \rightarrow \underline{u}$ up to a subsequence (denoted by $(u_i)$ itself) in $L^p(\Omega)$ and $L^\frac{pq}{q-1}(\Omega)$. Thus $\|\underline{u}\|_{L^p(\Omega)}=1$. From the weak lower semi-continuity of the norm $\|\cdot\|_{\mathbb{X}_0^{s,p}(\Omega)}$, we obtain
    \begin{equation}\label{F-T1-4}
        \|\underline{u}\|_{\mathbb{X}_0^{s,p}(\Omega)}\leq \liminf\limits_{i \rightarrow \infty}\|u_i\|_{\mathbb{X}_0^{s,p}(\Omega)}.
    \end{equation}
 Note that $\lambda_1(\underline{V})\geq L_1$ since $\underline{V}\in \mathcal{A}$. From the fact that $\Omega$ is bounded and using the dominated convergence theorem, it is easy to see that $|u_i|^p \rightarrow |u|^p$ in $L^{\frac{q}{q-1}}(\Omega)$ as $i \rightarrow\infty$. Also, let $\|V_i\|_{L^q(\Omega)}\leq M$ for all $i \in \mathbb{N}$. Using the weak convergence of $(V_i)$ to $V$ in $L^q(\Omega)$ and applying the H\"older inequality, it follows that
 \begin{align*}
     \left|\int_\Omega V_i |u_i|^p-\int_\Omega V|u|^pdx\right|&\leq \int_\Omega |V_i|\cdot\left||u_i|^p-|u|^p\right|dx+\int_\Omega |V_i-V|\cdot|u|^pdx\\
     &\leq M\left\||u_i|^p-|u|^p\right\|_{L^\frac{q}{q-1}(\Omega}+\int_\Omega |V_i-V|\cdot|u|^pdx\\
     &\rightarrow 0 \text{ as } i \rightarrow \infty.
 \end{align*}
 Thus, taking the limit infimum as $i \rightarrow \infty$  in \eqref{F-T1-1}, using \eqref{F-T1-4} and Theorem \ref{F-E.T1}, we get
    \begin{equation*}
        \lambda_1(\underline{V})\leq \|\underline{u}\|_{\mathbb{X}_0^{s,p}(\Omega)}^p+\int_\Omega \underline{V}|\underline{u}|^p dx\leq L_1\leq \lambda_1(\underline{V}).
    \end{equation*}
    Thus, $\underline{V}$ is a minimizer of $\lambda_1$ in $\mathcal{A}$ and $\underline{u}$ is an eigenfunction associated with $\lambda_1(\underline{V})$ for problem \eqref{F} with potential $V=\underline{V}$. By the choice of $\underline{V}$, we have
    $$\|\underline{u}\|_{\mathbb{X}_0^{s,p}(\Omega)}^p+\int_\Omega \underline{V}|\underline{u}|^p dx=\lambda_1(\underline{V}) \leq \lambda_1(V)\leq \|\underline{u}\|_{\mathbb{X}_0^{s,p}(\Omega)}^p+\int_\Omega V|\underline{u}|^p dx,$$
    for all $V\in \mathcal{A}$. Thus, we deduce
    $$\int_\Omega \underline{V}|\underline{u}|^p dx \leq \int_\Omega V|\underline{u}|^p dx, \ V\in \mathcal{A}.$$
    Now, assume that there exists $V\in \mathcal{A}$ such that
    $$\int_\Omega \underline{V}|\underline{u}|^p dx = \int_\Omega V|\underline{u}|^p dx.$$
    Then, we get
    \begin{equation*}
        \lambda_1(V) \leq \|\underline{u}\|_{\mathbb{X}_0^{s,p}(\Omega)}^p+\int_\Omega V|\underline{u}|^p dx=\lambda_1(\underline{V}).
    \end{equation*}
    By the choice of $\underline{V}$ and from Theorem \ref{F-E.T1}, we see that $\lambda_1(V)=\lambda_1(\underline{V})$ and $\underline{u}$ is an eigenfunction associated with $\lambda_1(V)$ for \eqref{F} with potential $V$. Using Theorem \ref{F-E.T3}, we may assume without loss of generality that $\underline{u}>0$ a.e. in $\Omega$. Hence, using the weak formulation of \eqref{F} in Definition \ref{F-WF}, we get
    \begin{equation*}
        H_{s,p}(\underline{u},\phi)+\int_{\Omega}\underline{V}|\underline{u}|^{p-2}\underline{u}\phi dx=H_{s,p}(\underline{u},\phi)+\int_{\Omega}V|\underline{u}|^{p-2}\underline{u} \phi dx,  
    \end{equation*} 
    for all $\phi \in C_c^\infty(\Omega)$. This implies
    \begin{equation*}
        \int_{\Omega}\left(\underline{V}-V\right)|u|^{p-2}u\phi dx=0, \text{ for all }\phi \in C_c^\infty(\Omega).
    \end{equation*}
    Therefore, $(\underline{V}-V)|\underline{u}|^{p-2}\underline{u} \equiv 0$ a.e. in $\Omega$. Since $\underline{u}>0$ a.e. in $\Omega$, we get $\underline{V}=V$ a.e. in $\Omega$. i.e., $\underline{V}$ is a unique minimizer in $\mathcal{A}$ of the map
    $$V\mapsto \int_{\Omega}V|\underline{u}|^{p}dx.$$
    This completes the proof.
\hfill\qedsymbol{}\vspace{.2cm} \\
We conclude this section by proving Theorem \ref{F-T2}, which establishes the existence of a unique maximizer of $\lambda_1(\cdot)$ in any closed, bounded and convex subset of $L^q(\Omega)$.\vspace{.2cm} \\ 
\noindent {\it{\bf{Proof of Theorem} \ref{F-T2}.}}
    Let $L_2:=\sup\{\lambda_1(V): V\in \mathcal{A}\}$. Then, there exists a sequence $(V_i)$ in $\mathcal{A}$ such that $\lambda_1(V_i) \rightarrow L_2$ as $i \rightarrow \infty$. Proceeding similarly to the proof of Theorem \ref{F-T1}, we get a subsequence (denoted by $(V_i)$ itself) and a sequence of functions $(u_i)\subset \{u\in \mathbb{X}_0^{s,p}(\Omega): \|u\|_{L^p(\Omega)}=1\}$ such that
    \begin{equation*}
        \|u_i\|_{\mathbb{X}_0^{s,p}(\Omega)}^p+\int_\Omega V_i|u_i|^p dx=\lambda_1(V_i),
    \end{equation*}
    and $V_i \rightharpoonup \overline{V}$ in $L^q(\Omega)$. Also, $\mathcal{A}$ is weakly closed and therefore, $\overline{V}\in \mathcal{A}$. By Theorem \ref{F-ET}, we also obtain $u_i \rightharpoonup \overline{u}$ in $\mathbb{X}_0^{s,p}(\Omega)$ and $u_i \rightarrow \overline{u}$ in $L^p(\Omega)$ (up to a subsequence). Thus, $\|\overline{u}\|_{L^p(\Omega)}=1$. Observe that by Remark \ref{F-R2}, $v^p\in L^{\frac{q}{q-1}}(\Omega)$ for all $v\in \mathbb{X}_0^{s,p}(\Omega)$. As a consequence of the H\"older inequality, it easily follows that
    $$\lim\limits_{i \rightarrow \infty}\int_{\Omega}V_i|v|^p dx=\int_{\Omega}\overline{V}|v|^p dx.$$
    From Theorem \ref{F-E.T1}, for each $j\in\mathbb{N}$, there exists $v_j\in \mathbb{X}_0^{s,p}(\Omega)$ satisfying
    \begin{align}\label{F-T2-2}
        \lambda_1(\overline{V})+\frac{1}{j}&\geq \|v_j\|_{\mathbb{X}_0^{s,p}(\Omega)}^p+\int_\Omega \overline{V}|v_j|^p dx \nonumber\\
        &=\lim\limits_{i \rightarrow \infty}\left(\|v_j\|_{\mathbb{X}_0^{s,p}(\Omega)}^p+\int_{\Omega}V_i|v_j|^p dx\right)\nonumber\\
        &\geq \lim\limits_{i \rightarrow \infty}\lambda_1(V_i)=L_2.
    \end{align}
    Note that $\lambda_1(\overline{V})\leq L_2$, since $\overline{V}\in \mathcal(A)$. Hence, taking the limit as $j \rightarrow \infty$ in \eqref{F-T2-2}, we deduce that $\lambda_1(\overline{V})=L_2$. Next, we establish the uniqueness of the maximizer $\overline{V}$. Let $\overline{V}'$ be another maximizer of $\lambda_1$ in $\mathcal{A}$. Then, $\lambda_1(\overline{V}')=L_2$. Define
    $$V=\frac{1}{2}(\overline{V}+\overline{V}').$$
    Since $\mathcal{A}$ is convex, $V\in \mathcal{A}$. By Lemma \ref{F-L1}, $\lambda_1$ is concave. Thus, 
    \begin{equation}\label{F-T2-3}
        \lambda_1(V) \geq \frac{1}{2}(\lambda_1(\overline{V})+\lambda_1(\overline{V}'))=L_2.
    \end{equation}
    Thus, $V$ is also a maximizer of $\lambda_1$ in $\mathcal{A}$. Let $\overline{u}, \overline{u}', u \in \{v\in \mathbb{X}_0^{s,p}(\Omega): \|v\|_{L^p(\Omega)}=1\}$ be the eigenfunctions associated with the first eigenvalues of \eqref{F} with potential $\overline{V}, \overline{V}'$ and $V$ respectively. 
    Assume $\overline{u}\not\equiv \pm u$. Then by Theorem \ref{F-E.T4}, we have
    \begin{equation}\label{F-T2-4'}
        \|u\|_{\mathbb{X}_0^{s,p}(\Omega)}^p+\int_\Omega \overline{V}|u|^p dx>\lambda_1(\overline{V})=L_2.
    \end{equation}
    Since equality holds in \eqref{F-T2-3}, using \eqref{F-T2-4'}, we deduce
    \begin{align*}
        L_2&=\|u\|_{\mathbb{X}_0^{s,p}(\Omega)}^p+\int_\Omega V|u|^p dx\\
        &=\frac{1}{2}\left(\|u\|_{\mathbb{X}_0^{s,p}(\Omega)}^p+\int_\Omega \overline{V}|u|^p dx\right)+\frac{1}{2}\left(\|u\|_{\mathbb{X}_0^{s,p}(\Omega)}^p+\int_\Omega \overline{V}'|u|^p dx\right)\\
        &>L_2.
    \end{align*}
    This contradiction gives us $\overline{u}\equiv \pm u$. Without loss of generality, let $\overline{u}\equiv  u$. Similarly, we deduce $\overline{u}'\equiv u$. Then, since $\lambda_1(\overline{V})=\lambda_1(\overline{V}')=L_2$, we have
    \begin{equation*}
        H_{s,p}(u,\phi)+\int_{\Omega}\overline{V}|u|^{p-2}u\phi dx= L_2\int_{\Omega}|u|^{p-2}u \phi dx=H_{s,p}(u,\phi)+\int_{\Omega}\overline{V}'|u|^{p-2}u\phi dx,  
    \end{equation*} 
    for all $\phi \in C_c^\infty(\Omega)$. Proceeding similarly to the proof of Theorem \ref{F-T1}, we deduce that $\overline{V}=\overline{V}'$ a.e. in $\Omega$. This establishes the uniqueness of the maximizer, completing the proof.
\hfill\qedsymbol{}
\section{Optimization in a set of rearrangements}\label{F-s4}
\noindent Let $V\in  L^q(\Omega)$. Consider the set of rearrangements $\mathcal{V}$ given by \eqref{F-RV}. Let us denote by $\mathcal{A}_V$ the weak closure of $\mathcal{R}_V$ in $L^q(\Omega)$.
From \cite[Theorem 8.16]{R1987}, we have $\mathcal{R}_V\subset L^q(\Omega)$ and $\|W\|_{L^q(\Omega)}=\|V\|_{L^q(\Omega)}$ for all $W\in \mathcal{R}_V$. According to \cite[Lemma 2.1]{B1989}, $\mathcal{A}_V$ is convex and therefore strongly closed. Thus, $\mathcal{A}_V$ is a bounded, closed, and convex subset of $L^q(\Omega)$. Using Theorem \ref{F-T1} and Theorem \ref{F-T2}, we infer that there exist a unique maximizer $\overline{V}$ and a minimizer $\underline{V}$  for $\lambda_1(\cdot)$ in $\mathcal{A}_V$. Let $\underline{u}$ be the positive eigenfunction in $\mathbb{X}_0^{s,p}(\Omega)$ associated with $\lambda_1(\underline{V})$ for \eqref{F} with $V-\underline{v}$, satisfying $\|\underline{u}\|_{L^p(\Omega)}=1$. We now proceed to the main theorem in this section.
\begin{theorem}\label{F-T5}
    The minimizer $\underline{V}$ of $\lambda_1(\cdot)$ in $\mathcal{A}_V$ is an element of $\mathcal{R}_V$. i.e., $\underline{V}$ also minimizes $\lambda_1(\cdot)$ in $\mathcal{R}_V$.  Moreover, there exists a decreasing function $f:\mathbb{R} \rightarrow\mathbb{R}$ such that $\underline{V}=f \circ|\underline{u}|^p$.
\end{theorem}
\begin{proof}
    Define the map $F:\mathcal{A}_V \rightarrow\mathbb{R}$ by
    $$F(W)=\int_\Omega W|\underline{u}|^{p}dx.$$
    By \cite[Theorem 1, Theorem 4]{B1987}, we deduce that the maximum of $-F$ in $\mathcal{A}_V$ is attained at a function $V'\in \mathcal{R}_V$. Thus, the minimum of $F$ in $\mathcal{A}_V$ is attained at $V'$. Now, for any $W\in \mathcal{A}_V$, there exists a sequence $(V_i)$ in $\mathcal{R}_V$ such that $V_i \rightharpoonup W$ as $i \rightarrow\infty$. Hence, we have
    $$F(W)=\lim\limits_{i \rightarrow\infty} F(V_i) \geq F(V'),$$
    since each $V_i \in \mathcal{R}_V$. Thus, we deduce that $V'$ is a minimizer of $F$ in $\mathcal{A}_V$. However, by Theorem \ref{F-T1}, we have that $\underline{V}$ is the unique minimizer of $F$ in $\mathcal{A}_V$. Thus, $V'=\underline{V}$ a.e. in $\Omega$. Therefore, $\underline{V} \in \mathcal{R}_V$ and $\underline{V}$ minimize $\lambda_1(\cdot)$ in $\mathcal{R}_V$. Again, applying \cite[Theorem 5]{B1987} to $-F$, we obtain a decreasing function $f:\mathbb{R} \rightarrow\mathbb{R}$ satisfying $\underline{V}=f \circ|\underline{u}|^p$. The proof is now complete. 
\end{proof}

\section{Optimization in the closed unit ball}\label{F-s5}
\noindent In this section, we establish some properties of the optimizers of $\lambda_1(\cdot)$ in the closed unit ball in $L^q(\Omega)$, with the help of the results obtained in Section \ref{F-s3}. Throughout this section, we denote
\begin{align*}
    B(0,1)&=\{V\in L^q(\Omega): \ \|V\|_{L^q(\Omega)}<1\},\\
    \mathcal{A}&=\{V\in L^q(\Omega): \ \|V\|_{L^q(\Omega)}\leq 1\},  \\
    \text{and } B'&=\{V\in L^q(\Omega): \ \|V\|_{L^q(\Omega)}=1\}.
\end{align*}
In addition, we denote by $\overline{V}$, the unique maximizer, and by $\underline{V}$, a minimizer of $\lambda_1(\cdot)$ in $\mathcal{A}$. Since $\mathcal{A}$ is a bounded, closed, and convex set, the existence of $\overline{V}$ and $\underline{V}$ is guaranteed by Theorem \ref{F-T1} and Theorem \ref{F-T2}. Let $\overline{u}$ be the positive eigenfunction associated with $\lambda_1(\overline{V})$ and $\underline{u}$ be the positive eigenfunction associated with $\lambda_1(\underline{V})$ (where the positivity is ensured by Theorem \ref{F-E.T3}) for \eqref{F} with potentials $\overline{V}$ and $\underline{V}$, respectively. Before discussing the properties of $\overline{V}$ and $\underline{V}$, we need the following lemmas.
\begin{lemma}\label{F-L3}
    The functional $\lambda_1(\cdot)$ is continuous in $L^q(\Omega)$.
\end{lemma}
\begin{proof}
    Let $V_i \rightarrow V$ in $L^q(\Omega)$ as $i \rightarrow \infty$. Then, there exists $M>0$ such that $\|V_i\|_{L^q(\Omega)}\leq M$ for all $i \in \mathbb{N}$. For each $i$, let $u_i\in \mathbb{X}_0^{s,p}(\Omega)$ be the eigenfunction corresponding to $\lambda_1(V_i)$ for the problem \eqref{F} with potential $V_i$. Also, assume that $\|u_i\|_{L^p(\Omega)}=1$ and $u_i>0$ (which can be obtained by Theorem \ref{F-E.T3}) for all $i \in \mathbb{N}$. From Theorem \ref{F-E.T1}, we have
    \begin{equation}\label{F-L3-1}
        \|u_i\|_{\mathbb{X}_0^{s,p}(\Omega)}^p +\int_{\Omega}V_i|u_i|^p dx=\lambda_1(V_i), \text{ for all }i \in \mathbb{N}.
    \end{equation}
     Following arguments similar to the proof of \eqref{F-T1-2} and \eqref{F-T1-3} using Lemma \ref{F-L2}, we get two constants $C_1,C_2=C(\frac{1}{2})>0$ such that
    \begin{equation}\label{F-L3-2}
        \|u_i\|_{\mathbb{X}_0^{s,p}(\Omega)}^p \leq 2(\lambda_1(V_i)+C_2M).
    \end{equation}
    Indeed, from Theorem \ref{F-E.T1}, we have
    \begin{equation}\label{F-L3-3}
        \lambda_1(V_i)\leq \|u\|_{\mathbb{X}_0^{s,p}(\Omega)}^p +\int_{\Omega}V_i|u|^p dx,
    \end{equation}
    for all $i\in\mathbb{N}$ and all $u\in \mathbb{X}_0^{s,p}(\Omega)$ with $\|u\|_{L^p(\Omega)}=1$. By Remark \ref{F-R2}, $u^p\in L^{\frac{q}{q-1}}(\Omega)$. Using the H\"older inequality, we arrive at
    $$\lim\limits_{i \rightarrow \infty}\int_\Omega V_i |u|^p dx=\int_\Omega V |u|^p dx,$$
    for each $u\in \mathbb{X}_0^{s,p}(\Omega)$. Therefore, first taking the limit supremum as $i \rightarrow \infty$ and then taking the infimum over all $u\in \mathbb{X}_0^{s,p}(\Omega)$  in \eqref{F-L3-3}, we deduce
    \begin{align*}
        \limsup\limits_{i \rightarrow \infty}\lambda_1(V_i)&\leq \inf \left\{\|u\|_{\mathbb{X}_0^{s,p}(\Omega)}^p +\int_{\Omega}V|u|^p dx:\ u\in \mathbb{X}_0^{s,p}(\Omega), \ \|u\|_{L^p(\Omega)}=1\right\}\\
        &=\lambda_1(V).
    \end{align*}
    Choose a subsequence of $(V_i)$ (still denoted by $(V_i)$) such that
    \begin{equation}\label{F-L3-4}
        \lim\limits_{i \rightarrow \infty}\lambda_1(V_i)\leq \lambda_1(V).
    \end{equation}
    Combining \eqref{F-L3-2} and \eqref{F-L3-4}, we see that $(u_i)$ is a bounded sequence in $\mathbb{X}_0^{s,p}(\Omega)$. By Theorem \ref{F-RSB}, $\mathbb{X}_0^{s,p}(\Omega)$ is a reflexive, separable Banach space. Using this fact along with the compact embedding in Theorem \ref{F-ET} and Remark \ref{F-R2}, we deduce that $(u_i)$ has a subsequence (still denoted by $(u_i)$) such that $u_i \rightharpoonup u$ in $\mathbb{X}_0^{s,p}(\Omega)$ and $u_i \rightarrow u$ in both $L^p(\Omega)$ and $L^{\frac{pq}{(q-1)}}(\Omega)$. Thus, we have $\|u\|_{L^p(\Omega)}=1$. Observe that the norm $\|\cdot\|_{\mathbb{X}_0^{s,p}(\Omega)}$ is weak lower semi-continuous. Therefore, from the definition of $\lambda_1(V)$, we have
    \begin{equation}\label{F-L3-4'}
        \lambda_1(V)\leq \|u\|_{\mathbb{X}_0^{s,p}(\Omega)}^p +\int_{\Omega}V|u|^p dx\leq \liminf\limits_{i \rightarrow \infty}\left(\|u_i\|_{\mathbb{X}_0^{s,p}(\Omega)}^p +\int_{\Omega}V_i|u_i|^p dx\right).
    \end{equation}
    Taking the limit as $i \rightarrow \infty$ in \eqref{F-L3-1} and using \eqref{F-L3-4}, \eqref{F-L3-4'}, we obtain
    $$\lambda_1(V)\leq \|u\|_{\mathbb{X}_0^{s,p}(\Omega)}^p +\int_{\Omega}V|u|^p dx=\lim\limits_{i \rightarrow \infty}\lambda_1(V_i)\leq \lambda_1(V).$$
    Hence $u$ is an eigenfunction associated with $\lambda_1(V)$ and $\lambda_1(V)$ is the limit of $\lambda_1(V_i)$ as $i \rightarrow
    \infty$. 
\end{proof}
\begin{remark}\label{F-R3}
    The above proof also guaranties that if $V_i \rightarrow V$ in $L^q(\Omega)$ and $u_i\in \mathbb{X}_0^{s,p}(\Omega)$ with $\|u_i\|_{L^p(\Omega)}$ is a positive eigenfunction associated with $\lambda_1(V_i)$ for \eqref{F} with potential $V_i$, then the weak limit $u$ of $u_i$ in $\mathbb{X}_0^{s,p}(\Omega)$ is a positive eigenfunction associated with $\lambda_1(V)$ for the problem \eqref{F} with potential $V$. Also, $u_i$ converges weakly and by norm to $u$ in $ \mathbb{X}_0^{s,p}(\Omega)$, which gives
    $$u_i \rightarrow u \text{ as } i \rightarrow\infty \text{ in }  \mathbb{X}_0^{s,p}(\Omega).$$
\end{remark}
Note that for $V\in B'$, the tangent space of $B'$ at $V$ is given by
$$T(B';V):=\left\{F\in L^q(\Omega): \int_\Omega|V|^{q-2}VF dx=0\right\}.$$ 
Now, we prove the following lemma on a differentiability property of $\lambda_1(\cdot)$.
\begin{lemma}\label{F-L4}
    Let $V\in B', F\in T(B';V)$ and $\gamma:(-1,1) \rightarrow L^q(\Omega)$ be a differentiable curve with image contained in $B'$ that satisfies $\gamma(0)=V$ and $\gamma'(0)=F$. Then, the composite map $\lambda_1 \circ \gamma:(-1,1) \rightarrow \mathbb{R}$ given by $\lambda_1 \circ \gamma(t):=\lambda_1(\gamma(t))$ is continuous in $(-1,1)$ and differentiable at $0$. Also, if $u$ is a positive eigenfunction corresponding to $\lambda_1(V)$ for \eqref{F} with potential $V$ and $\|u\|_{L^p(\Omega)}=1$, then 
    $$(\lambda_1 \circ\gamma)'(0)=\int_\Omega F|u|^p dx.$$
\end{lemma}
\begin{proof}
    The continuity of $\lambda_1 \circ \gamma$ follows from the continuity of $\lambda_1$ and $\gamma$. For each $t\in (-1,1)$, denote by $u_t \ (\in \mathbb{X}_0^{s,p}(\Omega))$, the unique positive function associated with $(\lambda_1 \circ\gamma)(t)$ for \eqref{F} with potential $\gamma(t)$, such that $\|u_t\|_{L^p(\Omega)}=1$ and
    $$\lambda_1(\gamma(t))=\|u_t\|_{\mathbb{X}_0^{s,p}(\Omega)}^p+\int_\Omega \gamma(t)|u_t|^p dx.$$
    Now, observe that by the definition of $\lambda_1$, we have
    \begin{align}\label{F-L4-1}
        (\lambda_1\circ \gamma) (t)-(\lambda_1 \circ \gamma)(0)&=\lambda_1(\gamma(t))-\lambda_1(V) \nonumber\\
        &\leq \|u\|_{\mathbb{X}_0^{s,p}(\Omega)}^p+\int_\Omega \gamma(t)|u|^p dx-\left(\|u\|_{\mathbb{X}_0^{s,p}(\Omega)}^p+\int_\Omega V|u|^p dx\right) \nonumber\\
        &=\int_\Omega (\gamma(t)-V)|u|^p dx.
    \end{align}
    Similarly, we also get 
    \begin{equation}\label{F-L4-2}
        (\lambda_1\circ \gamma) (t)-(\lambda_1 \circ \gamma)(0)\geq \int_\Omega (\gamma(t)-V)|u_t|^p dx.
    \end{equation}
    Consider any sequence $(t_i)$ in $(0,1)$ such that $t_i \rightarrow 0$ and $(\lambda_1\circ \gamma)(t_i) \rightarrow (\lambda_1\circ \gamma)(0)$ as $i \rightarrow \infty$ and
    \begin{equation}\label{F-L4-3}
        \lim\limits_{i \rightarrow \infty} \frac{(\lambda_1\circ \gamma)(t_i)-(\lambda_1\circ \gamma)(0)}{t}=\liminf\limits_{i \rightarrow \infty} \frac{(\lambda_1\circ \gamma)(t)-(\lambda_1\circ \gamma)(0)}{t}.
    \end{equation}
    By the continuity of $\gamma$, we deduce that $\gamma( t_i )\rightarrow\gamma(0)=V$ as $i \rightarrow \infty$. Now, by Remark \ref{F-R3}, we deduce that there exists an eigenfunction $v\in X_0^{s,p(\Omega)}$ associated with $\lambda_1(V)$ for \eqref{F}, such that $u_{t_i} \rightarrow v$ in $\mathbb{X}_0^{s,p}(\Omega)$ (and in $L^p(\Omega)$ up to a subsequence. In addition, $\|v\|_{L^p(\Omega)=1}$ and $v>0$ a.e. in $\Omega$. Therefore, it follows from Theorem \ref{F-E.T4} that $v\equiv u$. Since $\gamma'(0)=F$, we have
    $$\frac{(\gamma)(t_i)-V}{t}|u_{t_i}|^p  \rightarrow F|u|^p \text{ a.e. in } \Omega.$$
    Observe that by Remark \ref{F-R2}, making use of the H\"older inequality and the boundedness of $\Omega$, there exist constants $C_1, C_2>0$ such that
    $$\lim\limits_{i \rightarrow \infty} \left(\frac{(\gamma(t_i)-V)}{t}|u_{t_i}|^p \right)\leq \left( \gamma'(0)+C_1\right)\left(u+C_2\right)=(F+C_1)(u+C_2)\in L^1(\Omega).$$
    Recalling the characterization of $\lambda_1$ obtained by Theorem \ref{F-E.T1}, from \eqref{F-L4-2}, \eqref{F-L4-3} and using the dominated convergence theorem, we obtain
    \begin{align}\label{F-L4-4}
        \liminf\limits_{t \rightarrow 0^+} \frac{(\lambda_1\circ \gamma)(t_i)-(\lambda_1\circ \gamma)(0)}{t} &\geq  \int_\Omega \lim\limits_{i \rightarrow \infty} \left(\frac{\gamma(t_i)-V}{t}|u_{t_i}|^p\right) dx \nonumber\\
        &=\int_\Omega F|u|^p dx.
    \end{align}
    Proceeding in a similar way using \eqref{F-L4-1}, we also get
    \begin{equation}\label{F-L4-5}
        \limsup\limits_{t \rightarrow0^+} \frac{(\lambda_1\circ \gamma)(t_i)-(\lambda_1\circ \gamma)(0)}{t} \leq \int_\Omega F|u|^p dx.
    \end{equation}
    Observe that by Remark \ref{F-R2}, $|u|^p$ is in the dual $L^\frac{q}{q-1}(\Omega)$ of $L^q(\Omega)$. Now, it follows from \eqref{F-L4-1} that
    \begin{align}\label{F-L4-6}
        \liminf\limits_{t \rightarrow 0^-}\frac{(\lambda_1\circ \gamma)(t_i)-(\lambda_1\circ \gamma)(0)}{t}&=-\limsup\limits_{t \rightarrow 0^-}\frac{(\lambda_1\circ \gamma)(t_i)-(\lambda_1\circ \gamma)(0)}{(-t)} \nonumber\\
        & \leq -\limsup\limits_{t \rightarrow 0^-} \int_\Omega\frac{\gamma(t)-V}{(-t)}|u|^p dx \nonumber\\
        &=\int_\Omega \liminf\limits_{t \rightarrow 0^-}\frac{\gamma(t)-V}{t}|u|^p dx\nonumber\\
        &=\int_\Omega F|u|^p dx.
    \end{align}
    Similarly, we deduce that
    \begin{equation}\label{F-L4-7}
        \limsup\limits_{t \rightarrow 0^-}\frac{(\lambda_1\circ \gamma)(t_i)-(\lambda_1\circ \gamma)(0)}{t} \geq \int_\Omega F|u|^p dx.
    \end{equation}
    Combining \eqref{F-L4-4}--\eqref{F-L4-7}, we obtain the desired result.
\end{proof}
Now, we move on to the main results.
\begin{theorem}\label{F-T3}
    The functions $\overline{V}$ and $\underline{V}$ are in $B'$ with $\overline{V}>0$ and $\underline{V}<0$ a.e. in $\Omega$. Also, we have
    \begin{equation*}
        \int_\Omega F|\overline{u}|^p dx=\int_\Omega G|\underline{u}|^p dx=0,
    \end{equation*}
    for all $F\in T(B',\overline{V})$ and $G\in T(B', \underline{V})$.
\end{theorem}
\begin{proof}
    If possible, assume that there exists a subset of $\Omega$ with positive measure such that $\overline{V}<0$. Then $|\overline{V}|\in \mathcal{A}$ and
    \begin{equation*}
        \|u\|_{X_0^{s,p}(\Omega}^p+\int_\Omega |\overline{V}| |u|^p dx > \|u\|_{X_0^{s,p}(\Omega}^p+\int_\Omega \overline{V} |u|^p dx=\lambda_1(\overline{V}),
    \end{equation*}
    for all $u\in \mathbb{X}_0^{s,p}(\Omega)$ with $\|u\|_{L^p(\Omega)}=1$. Taking the infimum over $\{u\in \mathbb{X}_0^{s,p}(\Omega): \|u\|_{L^p(\Omega)}=1\}$ and applying Theorem \ref{F-E.T1}, we get $\lambda_1(|\overline{V}|)>\lambda_1(\overline{V})$. This contradicts our assumption that $\overline{V}$ maximizes $\lambda_1$ in $\mathcal{A}$. Thus, $\overline{V}\geq 0$ a.e in $\Omega$. Observe that we have $\overline{V}>0$ in a set having positive measure, since $\overline{V}$ is a maximizer of $\lambda_1$ in $\mathcal{A}$. Now, if $\|\overline{V}\|_{L^q(\Omega)}<1$, then clearly, we have $\tilde{V}=\frac{\overline{V}}{\|\overline{V}\|_{L^q(\Omega)}} \in B'$ and $\tilde{V}>\overline{V}$ in a set of positive measure. Then, we have
    \begin{equation*}
        \|u\|_{X_0^{s,p}(\Omega}^p+\int_\Omega \tilde{V} |u|^p dx > \|u\|_{X_0^{s,p}(\Omega}^p+\int_\Omega \overline{V} |u|^p dx,
    \end{equation*}
    for all $u\in \mathbb{X}_0^{s,p}(\Omega)$ with $\|u\|_{L^p(\Omega)}=1$. This again contradicts our assumption that $\overline{V}$ is the maximizer of $\lambda_1$ in $\mathcal{A}$. Therefore, it follows that $\overline{V}\in B'$. Similarly, we deduce that $\underline{V}\leq 0$ a.e. in $\Omega$ and $\underline{V}\in B'$. Now, let $F \in T(B';\overline{V})$ and $G\in T(B', \underline{V})$. Consider two differentiable curves $\gamma_1,\gamma_2:(-1,1) \rightarrow L^q(\Omega)$ with images contained in $B'$ satisfying $\gamma_1(0)=\overline{V}, \ \gamma_2(0)=\underline{V}, \ \gamma_1'(0)=F$ and $\gamma_2'(0)=G$. Since $\overline{V}$ maximizes $\lambda_1 \circ \gamma_1$ in $(-1,1)$ and $\underline{V}$ minimizes $\lambda_1 \circ \gamma_2$ in $(-1,1)$, we have $\lambda_1 \circ \gamma_1)'(0)=(\lambda_1 \circ \gamma_2)'(0)=0$. Using Lemma \ref{F-L4}, we obtain
     \begin{equation*}
        \int_\Omega F|\overline{u}|^p dx=\int_\Omega G|\underline{u}|^p dx=0.
    \end{equation*}
    
\end{proof}

\begin{theorem}\label{F-T4}
    For any subset $A$ of $\Omega$ with $|A|>0$, we have $\overline{V}\not \equiv 0 \not \equiv \underline{V}$ in $A$ i.e., $\operatorname{supp}(\overline{V})=\operatorname{supp}(\underline{V})=\Omega$. Moreover, there exist constants $C_1>0$ and $C_2>0$ such that
    $$|\overline{u}|^p=C_1|\overline{V}|^{q-1} \text{ and } |\underline{u}|^p=C_1|\underline{V}|^{q-1} \text{ a.e. in } \Omega.$$
\end{theorem}
\begin{proof}
    The theorem follows by an argument similar to \cite[Proposition 3.10, Theorem 3.11]{FD2006}. However, for the sake of completeness, we present the proof. If possible, assume that $\operatorname{supp}(\overline{V}) \neq \Omega.$ Thus, there exists a set $U \subset \Omega$ with positive measure such that $U \cap \operatorname{supp}(\overline{V})=\emptyset$. Consider the characteristic function of $U$ denoted by $F=\chi_{U}$. Clearly, we have
    $$\int_\Omega |\overline{V}|^{q-2}\overline{V}F dx=0.$$
    Thus, $F\in T(B';\overline{v})$. From Theorem \ref{F-T3}, we deduce that
    $$0=\int_\Omega F |\overline{u}|^p dx=\int_U |\overline{u}|^p dx.$$
    Thus it follows that $\overline{u}\equiv 0$ a.e. in $U$, which contradicts the strict positivity of $\overline{u}$. Thus, $\operatorname{supp}(\overline{V})=\Omega$. Similarly, we deduce $\operatorname{supp}(\underline{V})=\Omega$. Now, let $U_1,\ U_2$ be two subsets of  $\Omega$ with positive measure. Then $\overline{u}$ is not equivalent to zero a.e. in both $U_1$ and $U_2$. Thus,
    $$\int_{U_1}|\overline{V}|^{q-1} dx\neq 0 \neq \int_{U_2}|\overline{V}|^{q-1} dx.$$
    Now, define the function $F$ by 
    $$F(x):=\frac{\chi_{U_1}\overline{V}}{\int_{U_1}|\overline{V}|^{q-1} dx}-\frac{\chi_{U_2}\overline{V}}{\int_{U_2}|\overline{V}|^{q-1} dx}, \ x\in \Omega.$$
    Then, we have 
    \begin{equation*}
        \int_{\Omega} |\overline{V}|^{q-2} \overline{V}F dx =\frac{\int_{U_1}|\overline{V}|^{q-1} dx}{\int_{U_1}|\overline{V}|^{q-1} dx}-\frac{\int_{U_2}|\overline{V}|^{q-1} dx}{\int_{U_2}|\overline{V}|^{q-1} dx}=1-1=0.
    \end{equation*}
    Therefore, $F\in T(B';\overline{V})$. By Theorem \ref{F-T3}, we obtain the following
    $$0=\int_\Omega F|\overline{u}|^p dx=\frac{\int_{U_1}|\overline{u}|^{q-1} dx}{\int_{U_1}|\overline{V}|^{q-1} dx}-\frac{\int_{U_2}|\overline{u}|^{q-1} dx}{\int_{U_2}|\overline{V}|^{q-1} dx}.$$
    Since $U_1, \ U_2$ are arbitrary subsets of $\Omega$ with positive measure, we get a constant $C_1>0$ such that
    \begin{equation}\label{F-T4-1}
        C_1=\frac{\int_{U}|\overline{u}|^{q-1} dx}{\int_{U}|\overline{V}|^{q-1} dx}, \text{ for all } U \subset\Omega, \ |U|>0.
    \end{equation}
    Consider the set $U \subset \Omega$ given by $U:=\left\{\left(|\overline{u}|^p-C_1|\overline{V}|^{q-1}\right)>0\right\}$. If possible, assume that $|U|>0$. Then, \eqref{F-T4-1} implies \begin{equation}\label{F-T4-2}
        \int_U \left(|\overline{u}(x)|^p-C_1|\overline{V}(x)|^{q-1}\right) dx=0.
    \end{equation}
  But, from the positivity of the integrand is $|\overline{u}(x)|^p-C_1|\overline{V}(x)|^{q-1}>0$ in $U$, we also get
  $$\int_U \left(|\overline{u}(x)|^p-C_1|\overline{V}(x)|^{q-1}\right) dx>0.$$
  This contradicts \eqref{F-T4-2} and hence $|U|=0$. Similarly, we also get $$\left|\left\{\left(|\overline{u}|^p-C_1|\overline{V}|^{q-1}\right)<0\right\}\right|=0.$$ Hence $C_1|\overline{V}|^{q-1}\equiv |\overline{u}|^p$ a.e. in $\Omega$. Using the above approach, one can conclude that $C_1|\underline{V}|^{q-1}\equiv |\underline{u}|^p$ a.e. in $\Omega$. 
\end{proof}
\noindent \textbf{Conflict of interest statement} On behalf of the authors, the corresponding author states that there is no conflict of interest.
\newline
\textbf{Data availability statement:} Data sharing does not apply to this article as no datasets were generated or analyzed during the current study.

\section*{Acknowledgement}
R. Lakshmi thanks the Ministry of Education, Government of India, for the financial assistance. S. Ghosh gratefully acknowledges financial support for this research under the ARG-MATRICS grant No. ANRF/ARGM/2025/001570/MTR, Anusandhan National Research Foundation (ANRF), Government of India. R.K. Giri acknowledges the DST-FIST program (Govt. of India) for providing financial support for setting up the research and computing lab facility at the Department of Mathematics, The LNM Institute of Technology, Jaipur, under the scheme “Fund for Improvement of Science and Technology” (FIST - No. SR/FST/MS-I/2018/24).


\begin{thebibliography}{10}

\bibitem{AH1998}
W.~Allegretto and Y.~X. Huang.
\newblock A {P}icone's identity for the {$p$}-{L}aplacian and applications.
\newblock {\em Nonlinear Anal.}, 32(7):819--830, 1998.

\bibitem{A2008}
S.~Amghibech.
\newblock On the discrete version of {P}icone's identity.
\newblock {\em Discrete Appl. Math.}, 156(1):1--10, 2008.

\bibitem{AH1987}
M.~S. Ashbaugh and E.~M. Harrell.
\newblock Maximal and minimal eigenvalues and their associated nonlinear equations.
\newblock {\em J. Math. Phys.}, 28(8):1770–1786, 1987.

\bibitem{BDG2023}
N.~Biswas, U.~Das, and M.~Ghosh.
\newblock On the optimization of the first weighted eigenvalue.
\newblock {\em Proc. Roy. Soc. Edinburgh Sect. A}, 153(6):1777--1804, 2023.

\bibitem{BH17} V. Bonnaillie-Noel, B. Helffer,
\newblock {\em Nodal and spectral minimal partitions-the state of the art in 2016}, 
\newblock {\em In Shape optimization and Spectral Theory}, pages 353-397, De Gruyter, Warsaw, 2017.

\bibitem{B2011}
H.~Brezis.
\newblock {\em Functional analysis, {S}obolev spaces and partial differential equations}.
\newblock Springer, New York, 2011.

\bibitem{BDD2022}
S.~Buccheri, J.~V. da~Silva, and L.~H. de~Miranda.
\newblock A system of local/nonlocal {$p$}-{L}aplacians: the eigenvalue problem and its asymptotic limit as {$p\to\infty$}.
\newblock {\em Asymptot. Anal.}, 128(2):149--181, 2022.

\bibitem{B1987}
G.~R. Burton.
\newblock Rearrangements of functions, maximization of convex functionals, and vortex rings.
\newblock {\em Math. Ann.}, 276(2):225--253, 1987.

\bibitem{B1989}
G.~R. Burton.
\newblock Rearrangements of functions, saddle points and uncountable families of steady configurations for a vortex.
\newblock {\em Acta Math.}, 163(3-4):291--309, 1989.

\bibitem{CM1990}
S.~J. Cox and J.~R. McLaughlin.
\newblock Extremal eigenvalue problems for composite membranes. {I}, {II}.
\newblock {\em Appl. Math. Optim.}, 22(2):153--167, 169--187, 1990.

\bibitem{CEP2009}
F.~Cuccu, B.~Emamizadeh, and G.~Porru.
\newblock Optimization of the first eigenvalue in problems involving the {$p$}-{L}aplacian.
\newblock {\em Proc. Amer. Math. Soc.}, 137(5):1677--1687, 2009.

\bibitem{CQ2009}
M.~Cuesta and H.~R. Quoirin.
\newblock A weighted eigenvalue problem for the $p$-{L}aplacian plus a potential.
\newblock {\em NoDEA Nonlinear Differential Equations Appl.}, 16(4):469--491, 2019.

\bibitem{PBR2018}
L.~Del~Pezzo, J.~Fernández~Bonder, and L.~López~Ríos.
\newblock An optimization problem for the first eigenvalue of the $p$-fractional laplacian.
\newblock {\em Math. Nachr.}, 291(4):632–651, 2018.

\bibitem{PB2010}
L.~M. Del~Pezzo and J.~Fernández~Bonder.
\newblock An optimization problem for the first weighted eigenvalue problem plus a potential.
\newblock {\em Proc. Amer. Math. Soc.}, 138(10):3551–3567, 2010.

\bibitem{DFR2019}
L.~M. Del~Pezzo, R.~Ferreira, and J.~D. Rossi.
\newblock Eigenvalues for a combination between local and nonlocal {$p$}-{L}aplacians.
\newblock {\em Fract. Calc. Appl. Anal.}, 22(5):1414--1436, 2019.

\bibitem{DD2012}
F.~Demengel and G.~Demengel.
\newblock {\em Functional spaces for the theory of elliptic partial differential equations}.
\newblock Springer, London; EDP Sciences, Les Ulis, 2012.

\bibitem{DKP2016}
A.~Di~Castro, T.~Kuusi, and G.~Palatucci.
\newblock Local behavior of fractional {$p$}-minimizers.
\newblock {\em Ann. Inst. H. Poincar\'e{} C Anal. Non Lin\'eaire}, 33(5):1279--1299, 2016.

\bibitem{NPV2012}
E.~Di~Nezza, G.~Palatucci, and E.~Valdinoci.
\newblock Hitchhiker's guide to the fractional {S}obolev spaces.
\newblock {\em Bull. Sci. Math.}, 136(5):521--573, 2012.


\bibitem{FD2006}
J.~Fernández~Bonder and L.~M. Del~Pezzo.
\newblock An optimization problem for the first eigenvalue of the $p$-laplacian plus a potential.
\newblock {\em Commun. Pure Appl. Anal.}, 5(4):675–690, 2006.

\bibitem{FP2014}
G.~Franzina and G.~Palatucci.
\newblock Fractional $p$-eigenvalues.
\newblock {\em Riv. Math. Univ. Parma (N.S.)}, 5(2):373--386, 2014.

\bibitem{AP1987}
J.~P. Garc\'{\i}a~Azorero and I.~Peral~Alonso.
\newblock Existence and nonuniqueness for the {$p$}-{L}aplacian: nonlinear eigenvalues.
\newblock {\em Comm. Partial Differential Equations}, 12(12):1389--1430, 1987.


\bibitem{BS2007}
F. Gesztesy, P. Deift, C. Galvez, P. Perry, and W. Schlag (Eds.) Spectral Theory and Mathematical Physics: A Festschrift in Honor of Barry Simon's 60th Birthday, Proceedings of Symposia in Pure Mathematics, Vol. 76, part 2, American Mathematical Society, Providence, RI, 2007.

\bibitem{G2025}
M.~Ghosh.
\newblock On the optimization of the first weighted eigenvalue of the fractional {L}aplacian.
\newblock {\em Math. Nachr.}, 298(10):3251--3271, 2025.


\bibitem{G18} G. Gilboa,
\newblock {\em Nonlinear Eigenproblems in Image Processing and Computer Vision}, 
\newblock Springer, xx+172 pages, 2018.



\bibitem{H1984}
E.~M. Harrell~II.
\newblock Hamiltonian operators with maximal eigenvalues.
\newblock {\em J. Math. Phys.}, 25(1):48--51, 1984.

\bibitem{K1955}
M.~G. Krein.
\newblock On certain problems on the maximum and minimum of characteristic values and on the {L}yapunov zones of stability.
\newblock {\em Amer. Math. Soc. Transl. (2)}, 1:163--187, 1955.

\bibitem{LG2025}
R.~Lakshmi and S.~Ghosh.
\newblock Mixed local and nonlocal eigenvalue problems in the exterior domain.
\newblock {\em Fract. Calc. Appl. Anal.}, 28(4):1831–1866, 2025.

\bibitem{LGG2026}
R.~Lakshmi, R.~Kr. Giri, and S.~Ghosh.
\newblock A weighted eigenvalue problem for mixed local and nonlocal operators with potential.
\newblock {\em Math. Nachr.}, 299(2):367--396, 2026.

\bibitem{L2023}
G.~Leoni.
\newblock {\em A First Course in Fractional {S}obolev spaces}, volume 229.
\newblock American Mathematical Society, Providence, RI, 2023.

\bibitem{LL2014}
E.~Lindgren and P.~Lindqvist.
\newblock Fractional eigenvalues.
\newblock {\em Calc. Var. Partial Differential Equations}, 49(1-2):795--826, 2014.

\bibitem{L1990}
P.~Lindqvist.
\newblock On the equation {${\rm div}\,(|\nabla u|^{p-2}\nabla u)+\lambda|u|^{p-2}u=0$}.
\newblock {\em Proc. Amer. Math. Soc.}, 109(1):157--164, 1990.

\bibitem{MPV2013}
M. Marras, G. Porru and S. Vernier-Piro.
\newblock Optimization problems for eigenvalues of {$p$}-{L}aplace equations.
\newblock {\em J. Math. Anal. Appl.}, 398(2):766--775, 2013.

\bibitem{R1987}
W.~Rudin.
\newblock {\em Real and Complex Analysis, Third Edition}.
\newblock McGraw-Hill Book Co., New York, 1987.

\bibitem{BS2022}
B. Simon. 
\newblock Twelve tales in mathematical physics: An expanded Heineman prize lecture. 
\newblock {\em Journal of Mathematical Physics} 63(2), Paper no. 021101, 2022. 
\end{thebibliography}

\end{document}